\documentclass{amsart}

\usepackage{amssymb,amsmath,amsthm,latexsym,booktabs,qtree}

\usepackage{graphicx}
\usepackage{epstopdf}
\usepackage{hyperref}

\usepackage{tikz}
\usetikzlibrary{arrows}
\usetikzlibrary{calc}
\usetikzlibrary{positioning}

\theoremstyle{definition}
\newtheorem{definition}{Definition}[section]
\newtheorem{example}[definition]{Example}
\newtheorem{remark}[definition]{Remark}

\theoremstyle{plain}
\newtheorem{lemma}[definition]{Lemma}
\newtheorem{proposition}[definition]{Proposition}
\newtheorem{theorem}[definition]{Theorem}

\newcommand{\circh}{\circ} 
\newcommand{\circv}{\bullet}

\newcommand{\Size}{6 mm} 

\tikzset{Square/.style={
  inner sep=0pt,
  minimum width=\Size,
  minimum height=\Size,
  draw=black,
  fill=none,
  align=center
  }}

\allowdisplaybreaks

\begin{document}

\title{Permutation of elements in double semigroups}

\author{Murray Bremner}

\address{Department of Mathematics and Statistics, University of Saskatchewan, Canada}

\email{bremner@math.usask.ca}

\author{Sara Madariaga}

\address{Department of Mathematics and Statistics, University of Saskatchewan, Canada}

\email{madariaga@math.usask.ca}




\begin{abstract}
Double semigroups have two associative operations $\circh, \circv$ related by the interchange relation:
$( a \circv b ) \circh ( c \circv d ) \equiv ( a \circh c ) \circv ( b \circh d )$.
Kock \cite{Kock2007} (2007) discovered a commutativity property in degree 16 for double semigroups:
associativity and the interchange relation combine to produce permutations of elements.
We show that such properties can be expressed in terms of cycles in directed graphs with edges labelled
by permutations.
We use computer algebra to show that 9 is the lowest degree for which commutativity occurs,
and we give self-contained proofs of the commutativity properties in degree 9.
\end{abstract}

\maketitle

\thispagestyle{empty}


\section{Introduction}

\begin{definition}
A \textbf{double semigroup} is a set $S$ with two associative binary operations $\circv, \circh$ satisfying
the \textbf{interchange relation} for all $a, b, c, d \in S$:
  \begin{equation}
  \label{interchange}
  \tag{$\boxplus$}
  ( a \circv b ) \circh ( c \circv d ) \equiv ( a \circh c ) \circv ( b \circh d ).
  \end{equation}
The symbol $\equiv$ indicates that the equation holds for all values of the variables.
\end{definition}

We interpret $\circh$ and $\circv$ as horizontal and vertical compositions,
so that \eqref{interchange} expresses the equivalence of two decompositions of a square array:
  \[
  ( a \circh b ) \circv ( c \circh d )
  \equiv 
  \begin{array}{c}
  \begin{tikzpicture}[draw=black, x=\Size, y=\Size]
    \node [Square] at ($(0,0)$) {$a$};
    \node [Square] at ($(1,0)$) {$b$};
    \node [Square] at ($(0,-1.2)$) {$c$};
    \node [Square] at ($(1,-1.2)$) {$d$};
  \end{tikzpicture}
  \end{array}
   \equiv 
  \begin{array}{c}
  \begin{tikzpicture}[draw=black, x=\Size, y=\Size]
    \node [Square] at ($(0,0)$) {$a$};
    \node [Square] at ($(1,0)$) {$b$};
    \node [Square] at ($(0,-1)$) {$c$};
    \node [Square] at ($(1,-1)$) {$d$};
  \end{tikzpicture}
  \end{array}
   \equiv 
  \begin{array}{c}
  \begin{tikzpicture}[draw=black, x=\Size, y=\Size]
    \node [Square] at ($(0,0)$) {$a$};
    \node [Square] at ($(1.2,0)$) {$b$};
    \node [Square] at ($(0,-1)$) {$c$};
    \node [Square] at ($(1.2,-1)$) {$d$};
  \end{tikzpicture}
  \end{array}
   \equiv
  ( a \circv c ) \circh ( b \circv d ).
  \]
This interpretation of the operations extends to any double semigroup monomial,
producing what we call the geometric realization of the monomial.

The interchange relation originated in homotopy theory and higher categories; 
see Mac Lane \cite[(2.3)]{MacLane1963} and \cite[\S XII.3]{MacLane1998}.
It is also called the Godement relation by some authors; see Simpson \cite[\S2.1]{Simpson2012}.
If both operations are unital, with the same unit 1, then the Eckmann-Hilton argument \cite{EckmannHilton1961} 
shows that the operations are equal and commutative, even without the assumption of associativity.
If one allows the operations to have different units, then a similar argument shows that
$1_\circh = 1_\circv$; see Brown \cite[\S 4]{Brown1982}.
We therefore assume that double semigroups are non-unital.

\subsection*{Commutativity properties}

This paper is motivated by a commutativity property for double semigroups discovered in 2007 by
Kock \cite[Proposition 2.3]{Kock2007}:
in degree 16, associativity and the interchange relation combine to produce permutations of variables
in two monomials with the same placement of parentheses and choice of operations.
In algebraic notation, Kock proved the following identity; note the transposition of $f$ and $g$:
  \begin{equation}
  \label{kockidentity}
  \tag{K}
  \begin{array}{l}
  ( a \circh b \circh c \circh d )
  \circv
  ( e \circh f \circh g \circh h )
  \circv
  ( i \circh j \circh k \circh \ell )
  \circv
  ( m \circh n \circh p \circh q )
  \equiv
  \\
  ( a \circh b \circh c \circh d )
  \circv
  ( e \circh g \circh f \circh h )
  \circv
  ( i \circh j \circh k \circh \ell )
  \circv
  ( m \circh n \circh p \circh q ).
  \end{array}
  \end{equation}
The geometric realization of this identity has the following form:
\vspace{1mm}
  \[
  \begin{array}{c}
  \begin{tikzpicture}[draw=black, x=\Size, y=\Size]
    \node [Square] at ($(0,0)$) {$a$};
    \node [Square] at ($(1,0)$) {$b$};
    \node [Square] at ($(2,0)$) {$c$};
    \node [Square] at ($(3,0)$) {$d$};
    \node [Square] at ($(0,-1)$) {$e$};
    \node [Square] at ($(1,-1)$) {$f$};
    \node [Square] at ($(2,-1)$) {$g$};
    \node [Square] at ($(3,-1)$) {$h$};
    \node [Square] at ($(0,-2)$) {$i$};
    \node [Square] at ($(1,-2)$) {$j$};
    \node [Square] at ($(2,-2)$) {$k$};
    \node [Square] at ($(3,-2)$) {$\ell$};
    \node [Square] at ($(0,-3)$) {$m$};
    \node [Square] at ($(1,-3)$) {$n$};
    \node [Square] at ($(2,-3)$) {$p$};
    \node [Square] at ($(3,-3)$) {$q$};
  \end{tikzpicture}
  \end{array}
  \equiv
  \begin{array}{c}
  \begin{tikzpicture}[draw=black, x=\Size, y=\Size]
    \node [Square] at ($(0,0)$) {$a$};
    \node [Square] at ($(1,0)$) {$b$};
    \node [Square] at ($(2,0)$) {$c$};
    \node [Square] at ($(3,0)$) {$d$};
    \node [Square] at ($(0,-1)$) {$e$};
    \node [Square] at ($(1,-1)$) {$g$};
    \node [Square] at ($(2,-1)$) {$f$};
    \node [Square] at ($(3,-1)$) {$h$};
    \node [Square] at ($(0,-2)$) {$i$};
    \node [Square] at ($(1,-2)$) {$j$};
    \node [Square] at ($(2,-2)$) {$k$};
    \node [Square] at ($(3,-2)$) {$\ell$};
    \node [Square] at ($(0,-3)$) {$m$};
    \node [Square] at ($(1,-3)$) {$n$};
    \node [Square] at ($(2,-3)$) {$p$};
    \node [Square] at ($(3,-3)$) {$q$};
  \end{tikzpicture}
  \end{array}
  \]
This identity can be interpreted as a generalized Eckmann-Hilton argument, with the squares around
the border providing the necessary room to manoeuvre that is provided by the units in the original case.

For further developments based on Kock's work, see DeWolf \cite{DeWolf2013}.
In particular, Selinger \cite[Prop.~3.2.4]{DeWolf2013} has used a similar argument in degree 10 to prove 
that in cancellative double semigroups the two operations are equal.
For generalizations of the interchange relation in universal algebra, see Padmanabhan and Penner \cite{PP2006}.

\subsection*{Graph interpretation}

The present paper puts the earlier results in perspective by developing a general theory of commutativity 
properties in double semigroups.
Our underlying point of view is a graph interpretation which allows systematic reasoning about the problem.
We show that commutativity properties can be expressed in terms of cycles in directed graphs 
with edges labelled by permutations.
For degree $n$, the set of vertices of the directed graph $F(n)$ consists of all multilinear monomials $m$ 
of degree $n$ in the free double semigroup, and the set of edges $m \to m'$ of $F(n)$ consists of all 
consequences of the interchange relation stating that $m$ can be converted into $m'$ by one application 
of \eqref{interchange}.

The symmetric group $S_n$ acts on $F(n)$ by permuting the \emph{positions} of the variables in the monomials, 
and produces the quotient graph $G(n)$ with the covering projection $p\colon F(n) \to G(n)$.
The directed graph $G(n)$ consists of all association types (composable configurations involving $n$ 
elements in a double semigroup) connected by moves which represent applications of the interchange relation 
\eqref{interchange} and are labelled by the corresponding permutations.  
The key point is that a commutativity property can now be interpreted as a directed cycle in $G(n)$ for which 
the product of the corresponding permutations is nontrivial.

Our hope is that the conceptual insight provided by this graph interpretation may lead to further developments 
of the theory of double semigroups.

\subsection*{Outline of the paper}

Section \ref{graphtheorysection} provides the precise definitions for our application of
graph theory to the problem of commutativity in double semigroups.
Section \ref{section45678} describes our search for nontrivial cycles in degree $n \le 8$ 
using the computer algebra system Maple, with examples to illustrate our methods.
For every cycle in these degrees the product of the edge permutations is the identity, and so there are no
commutativity properties.
Section \ref{sectiondegree9} extends our computations to degree 9 and presents a complete list of connected
components of $G(n)$ which contain nontrivial cycles.
For each of these 16 components, we describe the commutativity properties both algebraically and geometrically,
and provide self-contained proofs.
(None of the corresponding geometric realizations is the $3 \times 3$ square.)
Section \ref{sectionconclusion} contains some concluding remarks and suggestions for further research.


\section{The covering map of directed graphs $p\colon F(n) \to G(n)$} \label{graphtheorysection}

We fix an integer $n \ge 1$ and a set $X = \{ a_1, \dots, a_n \}$ of $n$ indeterminates.

\begin{definition} \label{definitionassociationtypes}
We write $VG(n)$ for the set of \textbf{association types} in degree $n$: the possible placements 
of balanced parentheses and binary operation symbols in a monomial of degree $n$ with two associative 
operations $\{ \circh, \circv \}$.
\end{definition}

In Definition \ref{definitionassociationtypes} we do not consider the interchange identity, only the 
two associativities, so we are dealing with the multiplicative structure underlying two-associative algebras 
in the sense of Loday \cite{Zinbiel2012}.
The arguments in the monomial are not specified, but their positions will be identified from left to right
with the natural numbers $1, \dots, n$.
If we are careful, we can identify the arguments $1, \dots, n$ with the indeterminates 
$a_1, \dots, a_n$ but then we must remember that permutations in $S_n$ act on the positions and not
on the subscripts of the indeterminates; this is especially important when we compose permutations.

\begin{definition} \label{definitionmultilinearmonomials}
We write $VF(n)$ for the set of \textbf{multilinear monomials} in degree $n$; they are obtained from 
the association types by inserting the indeterminates $a_1, \dots, a_n$ from left to right and applying 
all permutations $\sigma \in S_n$.
In this case (and only this case) it does not matter whether we apply $\sigma$ to the positions or 
to the subscripts of the indeterminates, so after applying $\sigma$ the underlying indeterminates 
from left to right are $a_{\sigma(1)}, \dots, a_{\sigma(n)}$.
\end{definition}

\begin{definition} \label{definitionforgetfulmap}
The \textbf{forgetful map} $p\colon VF(n) \to VG(n)$ assigns to a multilinear monomial $m$ 
its association type $p(m)$ by ignoring the permutation of the indeterminates.
As before, if we are careful, we may say that $p(m)$ is obtained from $m$ by replacing 
$a_{\sigma(1)}, \dots, a_{\sigma(n)}$ by the identity permutation $a_1, \dots, a_n$.
\end{definition}

The sets $VF(n)$ and $VG(n)$ are the vertex sets of the directed graphs $F(n)$ and $G(n)$ which 
will be defined shortly; the forgetful map $p$ is clearly surjective.
First we need to determine the sizes of the sets $VF(n)$ and $VG(n)$.

\begin{definition} \label{smalldefinition}
The \textbf{small Schr\"oder number} $t(n)$ is the number of distinct ways to insert balanced 
parentheses into a sequence of $n$ arguments such that each pair of parentheses encloses two or more arguments 
or parenthesized subsequences; parentheses enclosing a single argument are not allowed.
More formally, as in Stanley \cite[p.~345]{Stanley1997}, we may recursively define \emph{bracketings} 
as follows: 
the symbol $x$ is a bracketing, and for $k \ge 2$ the symbol $( w_1 \cdots w_k )$ is a bracketing 
where each $w_i$ is a bracketing for $1 \le i \le k$.
Then $t(n)$ is the number of distinct bracketings which contain exactly $n$ occurrences of $x$.
\end{definition}

\begin{definition}
The \textbf{large Schr\"oder number} $T(n)$ is the number of paths in the $xy$-plane from the origin $(0,0)$
to the point $(n,n)$ that do not rise above the line $y = x$ and which use only steps north $(0,1)$, east 
$(1,0)$, or northeast $(1,1)$.
\end{definition}

The sequences $t(n)$ and $T(n)$ appeared in 1870 in the work of Schr\"oder \cite{Schroder1870} on 
combinatorial problems.
For a detailed discussion, and for other combinatorial interpretations of the small and large 
Schr\"oder numbers, see Stanley \cite[p.~177; Exercise 6.39, p.~239]{Stanley1999}.
For further information about the Schr\"oder numbers and their history, see
Stanley \cite{Stanley1997}, Habsieger et al.~\cite{HKL1998} and Acerbi \cite{Acerbi2003}.

\begin{lemma} \label{lemmadoubling}
The small and large Schr\"oder numbers are related by the equations 
  \[
  T(1) = t(1) = 1, \qquad\qquad T(n) = 2 \, t(n) \qquad (n > 1).
  \]
\end{lemma}

\begin{proof}
This remarkably simple relation between $t(n)$ and $T(n)$ is not at all obvious from the
definitions.
For proofs, see Shapiro and Sulanke \cite{ShapiroSulanke2000} and Deutsch \cite{Deutsch2001}.
\end{proof}

\begin{lemma} \label{lemmaschroeder}
We have the following formula for the large Schr\"oder numbers:
  \[
  |T(1)| = 1,
  \qquad\qquad
  |T(n)| = \frac{1}{n} \sum_{i=1}^n 2^i \binom{n}{i} \binom{n}{i{-}1} \qquad (n > 1).
  \]
The first 10 values are as follows:
  \[
  \begin{array}{ccccccccccc}
  n &\; 1 &\; 2 &\; 3 &\; 4 &\; 5 &\; 6 &\; 7 &\; 8 &\; 9 &\; 10 \\
  T(n) &\; 1 &\; 2 &\; 6 &\; 22 &\; 90 &\; 394 &\; 1806 &\; 8558 &\; 41586 &\; 206098
  \end{array}
  \]
\end{lemma}

\begin{proof}
For the recurrence relation satisfied by the Schr\"oder numbers, 
see Foata and Zeilberger \cite{FoataZeilberger1997};
for the generating function, see Stanley \cite{Stanley1997}; 
and for the expression in terms of binomial coefficients, see Coker \cite{Coker2004}.
For further information and references, see sequences A001003 and A006318 in the OEIS (oeis.org).
\end{proof}

\begin{remark} \label{treeremark}
We regard each multilinear monomial of degree $n \ge 2$ as a product of factors using 
one of the operations; each of these factors of degree $\ge 2$ is a product of factors using 
the other operation; and so on.
Thus a multilinear monomial can be represented as a planar rooted tree with $n$ leaves 
labelled by some permutation of $a_1, \dots, a_n$ and with each internal node labelled by an operation
symbol $\circh$, $\circv$.
Following any path from the root to a leaf, the operation symbols alternate: the same operation symbol 
does not occur twice consecutively.
We write $Y(m)$ for the tree corresponding to the multilinear monomial $m$.
For example, if $m$ is the left side of Kock's identity \eqref{kockidentity} then $Y(m)$ has this form:
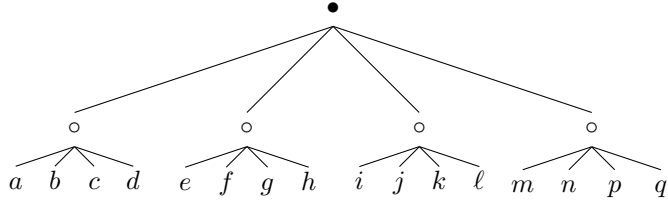
\[
\Tree [ 
.$\circv$ 
[ .$\circh$ $a$ $b$ $c$ $d$ ] 
[ .$\circh$ $e$ $f$ $g$ $h$ ] 
[ .$\circh$ $i$ $j$ $k$ $\ell$ ] 
[ .$\circh$ $m$ $n$ $p$ $q$ ] 
]
\]
The left sides of our commutativity properties in degree 9 correspond to the three trees in Figure
\ref{figure3trees} (see Theorems \ref{theorem3981}, \ref{theorem3989}, \ref{theorem3994} below).
\end{remark}

\begin{figure}[ht]
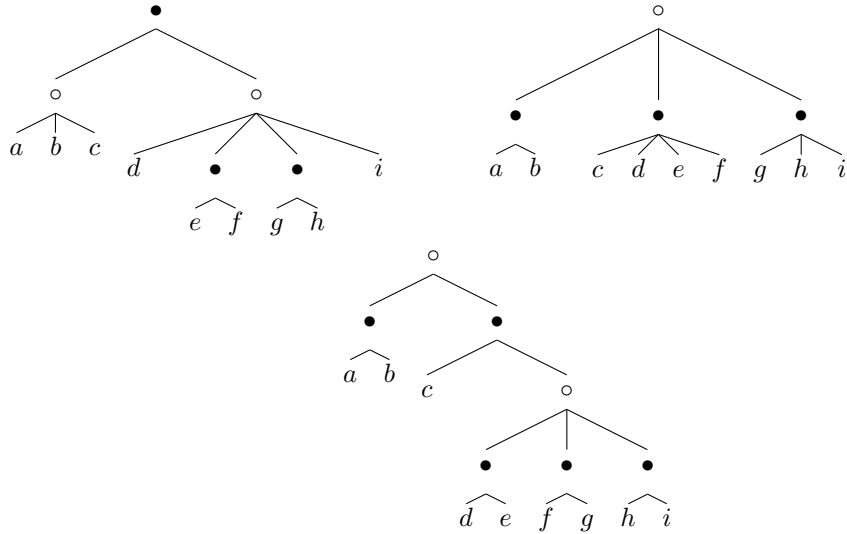

\[
\begin{array}{c}
\Tree [ 
.$\circv$ 
[ .$\circh$ $a$ $b$ $c$ ] 
[ .$\circh$ $d$ [ .$\circv$ $e$ $f$ ] [ .$\circv$ $g$ $h$ ] $i$ ]
]
\qquad\qquad
\Tree [ 
.$\circh$ 
[ .$\circv$ $a$ $b$ ] 
[ .$\circv$ $c$ $d$ $e$ $f$ ]
[ .$\circv$ $g$ $h$ $i$ ]
]
\\
\qquad\qquad\qquad
\Tree [ 
.$\circh$ 
[ .$\circv$ $a$ $b$ ]
[ .$\circv$ $c$ 
[ .$\circh$ [ .$\circv$ $d$ $e$ ] [ .$\circv$ $f$ $g$ ] [ .$\circv$ $h$ $i$ ] ]
]
]
\end{array}
\]
\caption{Trees representing left sides of new commutativity properties}
\label{figure3trees}
\end{figure}

\begin{lemma} \label{lemmaschroder}
We have $| VG(n) | = T(n)$ and $| VF(n) | = n! \, T(n)$.
\end{lemma}

\begin{proof}
Recall from Definition \ref{smalldefinition} the notion of bracketing of a sequence of $n$ arguments.
Suppose we now allow two colours of parentheses, say black and white, and impose two conditions:
(i) the parentheses in every balanced pair have the same colour, and (ii) the arguments immediately inside
every pair of black (resp.~white) parentheses are enclosed in white (resp.~black) parentheses.
This does not affect the case $n = 1$ in which $t(1) = 1$ since $x$ contains no parentheses.
For $n \ge 2$, the number of coloured bracketings is $2 \, t(n)$, since once a colour has been chosen for 
the outermost pair of parentheses, conditions (i) and (ii) imply that the colours of the remaining pairs 
of parentheses are completely determined.
Now observe that the placements of black and white parentheses are in bijection with the association types
for two-associative algebras: we merely replace a pair of white (resp.~black) parentheses around certain
arguments by a sequence of white (resp.~black) operation symbols $\circh, \circv$ between the consecutive
arguments.

If we regard association types as bipartite trees as in Remark \ref{treeremark}, then we can express the 
same concept in different words: the association types are bipartite planar rooted trees, so once a colour 
is chosen for the root vertex, the colours of all the other internal vertices are determined.
The conclusion of this reasoning is that the number of association types is equal to $2 \, t(n)$ for all 
$n > 1$, and combining this with Lemma \ref{lemmadoubling}, we obtain $| VG(n) | = T(n)$.
Since the $n$ arguments $a_1, \dots, a_n$ of the multilinear monomial (equivalently, the $n$ leaves of 
the tree) can be assigned arbitrarily to the $n$ positions, we conclude that $| VF(n) | = n! \, T(n)$.
\end{proof}

\begin{figure}[ht]
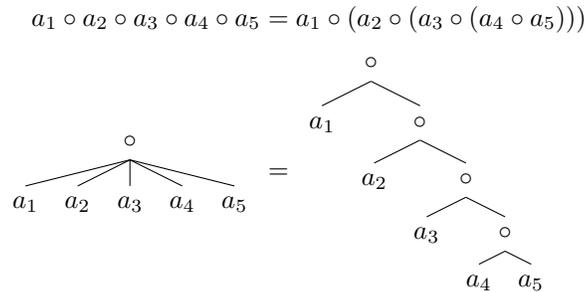

\begin{align*}
a_1 \circh a_2 \circh a_3 \circh a_4 \circh a_5
&=
a_1 \circh ( a_2 \circh ( a_3 \circh ( a_4 \circh a_5 ) ) )
\\
\begin{array}{c}
\Tree [ .$\circh$ $a_1$ $a_2$ $a_3$ $a_4$ $a_5$ ]
\end{array}
&=
\begin{array}{c}
\Tree [ .$\circh$ $a_1$ [ .$\circh$ $a_2$ [ .$\circh$ $a_3$ [ .$\circh$ $a_4$ $a_5$ ] ] ] ]
\end{array}
\end{align*}
\caption{Unique factorization in the presence of associativity}
\label{normalassociativity}
\end{figure}

\begin{remark} \label{remarkunique}
When no ambiguity is possible, we omit the parentheses in an associative product.
However, for computational purposes, it is essential to have a unique way to represent each associative product.
In order to guarantee unique factorization in the presence of associativity, we assume that whenever
more than two factors are multiplied by the same operation, the parentheses are right-justified:
the monomial is fully parenthesized and only left multiplications are used.
This is a useful convention for computer programming but can be misleading conceptually, as the bipartite 
property is lost.
Figure \ref{normalassociativity} gives a simple example.
\end{remark}

\begin{definition} \label{definitiontotalorder}
With the convention of Remark \ref{remarkunique}, every association type $t$ 
of degree $n \ge 2$ can be written uniquely as either $t = t' \circh t''$ or $t = t' \circv t''$.
This permits us to define recursively a \textbf{total order on association types}.
The basis is the unique association type in degree 1.
Let $t_1$ and $t_2$ be association types of degrees $n_1$ and $n_2$ respectively;
then $t_1$ precedes $t_2$ (denoted $t_1 \prec t_2$) if and only if:
  \begin{itemize}
  \item[(i)]
  $n_1 < n_2$: 
  first we look at the number of leaves;
  or
  \item[(ii)]
  $n_1 = n_2$ and $t_1 = t'_1 \circh t''_1$ and $t_2 = t'_2 \circv t''_2$: 
  next we look at the operation at the root with the convention that $\circh$ precedes $\circv$;
  or
  \item[(iii)]
  $n_1 = n_2$ and $t_1 = t'_1 \ast t''_1$ and $t_2 = t'_2 \ast t''_2$  with the same operation $\ast$, 
  and either $t'_1 \prec t'_2$, or $t'_1 = t'_2$ and $t''_1 \prec t''_2$:
  finally we use recursion on the unique factorization of each association type into two factors. 
  \end{itemize}
\end{definition}

We now have in place the information we need on the vertices of our directed graphs,
so we turn to the definition of the edges in terms of the interchange relation.

\begin{definition} \label{definitionconsequences}
The set $C(n)$ of \textbf{consequences in degree $n$ of the interchange relation}  
is defined inductively starting with $C(n) = \emptyset$ for $1 \le n \le 3$ and $C(4) = \{ m_1 \equiv m_2 \}$
where $m_1 \equiv m_2$ is the interchange relation \eqref{interchange}.
For $n \ge 4$ we consider a relation $R \in C(n)$ written as $m_1 \equiv m_2$ for multilinear monomials
$m_1, m_2 \in VF(n)$.
The consequences of $R$ in degree $n+1$ are defined as follows:
  \begin{itemize}
  \item[(i)]
For each $i \in \{1,\dots,n\}$ and $\ast \in \{\circh,\circv\}$ we substitute $a_i \ast a_{n+1}$ for $a_i$
in $m_1$ and $m_2$ and obtain these $2n$ relations in degree $n+1$:
  \[
  m_1 \big|_{a_i = a_i \ast a_{n+1}} \equiv \, m_2 \big|_{a_i = a_i \ast a_{n+1}}
  \]
  \item[(ii)]
For each $\ast \in \{\circh,\circv\}$ we right- or left-multiply $m_1$ and $m_2$ by $a_{n+1}$ and obtain 
these 4 relations in degree $n+1$:
  \[
  m_1 \ast a_{n+1} \,\equiv\, m_2 \ast a_{n+1},
  \qquad
  a_{n+1} \ast m_1 \,\equiv\, a_{n+1} \ast m_2.
  \]
  \end{itemize}
We write $c(R)$ for the set of all consequences of $R$ in degree $n+1$, and define
  \[
  C(n+1) = \bigcup_{R \in C(n)} c(R).
  \]
It follows immediately from (i) and (ii) that for $n \ge 4$, every element of $C(n)$ has the form 
$m_1 \equiv m_2$ where $m_1$ and $m_2$ are multilinear monomials of degree $n$.
\end{definition}

\begin{remark}
Definition \ref{definitionconsequences} implies that
  \[
  | C(n) | = \frac{ 2^{n+1} (n{+}1)! }{ 2^5 5! } \qquad (n \ge 4).
  \]
Since $|C(n)|$ grows super-exponentially, for computational purposes it is important to reduce the number 
of consequences as much as possible.
Consider $R \in C(n)$ written as $m_i \equiv m_j$ where multilinear monomials $m_i, m_j$ 
have association types $t_i, t_j$ and permutations $\sigma_i, \sigma_j$ of the indeterminates;
we may assume $t_i \prec t_j$.
If we apply the same permutation in $S_n$ to both $m_i$ and $m_j$ then we obtain an equivalent relation.
In particular, applying $\sigma_i^{-1}$ gives a relation in which the underlying permutation of 
$\sigma_i^{-1} m_i$ is the identity permutation $\iota$;
this result is denoted $N(R)$ and called the normal form of $R$:
  \[
  \overbrace{a_1 \, a_2 \, \cdots \, a_{n-1} \, a_n} ^{\text{association type $t_i$}}
  \; \equiv \;
  \overbrace{a_{\sigma_{ij}(1)} a_{\sigma_{ij}(2)} \cdots
  a_{\sigma_{ij}(n-1)} a_{\sigma_{ij}(n)}}^{\text{association type $t_j$}}
  \qquad
  (\sigma_{ij} = \sigma_i^{-1} \sigma_j).
  \]
We write $NC(n) = \{ \, N(R) \mid R \in C(n) \, \}$ for the set of normal forms with repetitions eliminated; 
this significantly reduces the total number as $n$ increases:
  \[
  \begin{array}{l|rrrrrr}
  n &\quad 
  4 & 5 & 6 &\quad 7 &\quad 8 &\quad 9 \\ \midrule
  |C(n)| &\quad 
  1 & 12 & 168 &\quad 2688 &\quad 48384 &\quad 967680 \\
  |NC(n)| &\quad 
  1 & 12 & 98 &\quad 688 &\quad 4482 &\quad 28004 \\
  |NC(n)|/|C(n)| &\quad 
  1 & 1 & {\approx\,}0.583 &\quad {\approx\,}0.256 &\quad {\approx\,}0.0926 &\quad {\approx\,}0.0289
  \end{array}
  \]
\end{remark}

\begin{definition}
Each relation $R \in C(n)$, written as $m_i \equiv m_j$, gives rise to 
two directed edges between the multilinear monomials $m_i, m_j \in VF(n)$:
  \[
  m_i \xrightarrow{\quad e_R \quad} m_j,
  \qquad\qquad
  m_i \xleftarrow{\quad e'_R \quad} m_j.
  \]
The set $EF(n)$ of \textbf{edges of the graph} $F(n)$ consists of all these edges:
  \[
  EF(n) = \bigcup_{R \in C(n)} \{ \, e_R, \, e'_R \, \}.
  \]
More precisely, we write $\sigma_i, \sigma_j$ for the underlying permutations of $m_i, m_j$
and represent the directed edges obtained from $R$ as follows; 
the permutation labelling one edge is the inverse of the permutation labelling the other:
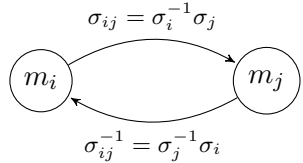
\[
\begin{array}{c}
\begin{tikzpicture}
[
->,
>=stealth',
shorten >=1pt,
auto,
node distance=3cm,
main node/.style={circle,draw,font=\large}
]
  \node[main node] (1) {$m_i$};
  \node[main node] (2) [right of=1] {$m_j$};
  \path[every node/.style={font=\small}]
  (1) edge [bend left] node [above] {$\sigma_{ij} = \sigma_i^{-1} \sigma_j$} (2)
  (2) edge [bend left] node [below] {$\sigma_{ij}^{-1} = \sigma_j^{-1} \sigma_i$} (1)
  ;
\end{tikzpicture}
\end{array}
\]
Since we compose permutations along a directed path, recall that we regard permutations as applying to
the positions of the indeterminates not their subscripts.
\end{definition}

We are now in a position to define the main object of study in this paper:

\begin{definition}
For $n \ge 1$, the \textbf{interchange graph} $F(n)$ is the directed graph with vertex set $VF(n)$ and 
edge set $EF(n)$.
The symmetric group $S_n$ acts freely on $VF(n)$ by permutation of the positions of the indeterminates in 
the monomials.
Since the edges in $EF(n)$ represent relations $m \equiv m'$ between monomials in $VF(n)$,
the action of $S_n$ extends compatibly to the graph $F(n)$.
Hence this action induces a \textbf{quotient graph} $G(n)$ and a \textbf{covering projection} 
$p\colon F(n) \to G(n)$.
\end{definition}

The vertices of the quotient graph $G(n)$ form the set $VG(n)$ of association types 
(Definition \ref{definitionassociationtypes}).
More precisely, these vertices are the orbits of the action of $S_n$ on $VF(n)$;
each orbit consists of all $n!$ multilinear monomials with the same association type.
The set $EG(n)$ of directed edges of $G(n)$ can be described explicitly without reference 
to the underlying set of indeterminates.
Each edge represents a single application of the interchange relation.
Regarding association types as trees, this amounts to identifying a subtree of the association type 
having the form of the left side of the interchange relation \eqref{interchange} and replacing it with 
the right side, keeping track of the permutations of the leaves in the four children $T_1, \dots, T_4$,
and noting the transposition of $T_2$ and $T_3$:
  
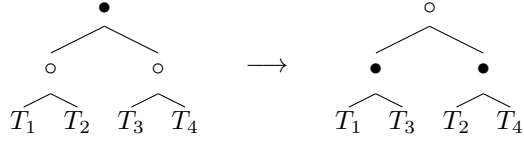
\[
  \begin{array}{c} \Tree [ .$\circv$ [ .$\circh$ $T_1$ $T_2$ ] [ .$\circh$ $T_3$ $T_4$ ] ] \end{array}
  \quad \longrightarrow \quad
  \begin{array}{c} \Tree [ .$\circh$ [ .$\circv$ $T_1$ $T_3$ ] [ .$\circv$ $T_2$ $T_4$ ] ] \end{array}
  \]
It is convenient to label the edges $EG(n)$ with permutations of the (positions of the) leaves,
but this is not necessary if we keep $F(n)$ in the picture.
To verify that $p$ is a covering projection we return to the explicit descriptions of the graphs: 
given an edge $t \to t'$ in $G(n)$ and a vertex $m$ in $F(n)$ lying over $t$ (thus $m$ is a 
multilinear monomial with association type $t$), there is a unique lifting of the edge in $G(n)$ to
an edge $m \to m'$ in $F(n)$ given by applying the permutation $\sigma$ corresponding to the edge 
$t \to t'$ in $G(n)$ to the vertex $m$ to obtain $m' = \sigma \cdot m$.
The unique path lifting property for the covering follows from the uniqueness of this lifting of edges.

\begin{example} \label{exampledegree4}
In the simplest case $n = 4$, the quotient graph $G(4)$ has 22 vertices: 20 are isolated and 
the other two are linked by a pair of edges forming a connected component which represents the interchange 
identity \eqref{interchange}:
\[
( a \circv b ) \circh ( c \circv d )
\,\,=
\begin{array}{c}
\begin{tikzpicture}
[
->,
>=stealth',
shorten >=1pt,
auto,
node distance=2cm,
main node/.style={circle,draw,font=\normalsize}
]
  \node[main node] (1) {8};
  \node[main node] (2) [right of=1] {18};
  \path[every node/.style={font=\normalsize}]
  (1) edge [bend left] node [above] {$(23)$} (2)
  (2) edge [bend left] node [below] {$(23)$} (1)
  ;
\end{tikzpicture}
\end{array}
=\,\,
( a \circh c ) \circv ( b \circh d )
\]
The numbers on the vertices are the index numbers of the corresponding association types in the total order 
of Definition \ref{definitiontotalorder}.
The permutations labelling the edges are written in cycle decomposition.
\end{example} 

\begin{example} \label{exampledegree5}
In degree 5, the graph $G(5)$ has 90 vertices and 12 edges, with 70 isolated vertices and 8 connected 
components: 4 each of sizes 2 and 3.
One of the 3-vertex components has association types 26, 30 and 68 as its vertices, where the positions 
are indicated by the alphabetical ordering of the indeterminates:
  \[
  (a \circv b) \circh (c \circv (d \circh e)), \qquad
  (a \circv (b \circh c)) \circh (d \circv e), \qquad
  (a \circh b) \circv (c \circh (d \circh e)).
  \]
To find the edges connecting these vertices, we start with the interchange relation \eqref{interchange}, 
separately substitute $d \leftarrow d \circh e$ and $b \leftarrow b \circh e$, and obtain two consequences:
  \begin{align*}
  ( a \circv b ) \circh ( c \circv ( d \circh e ) )
  &\equiv
  ( a \circh c ) \circv ( b \circh ( d \circh e ) ),
  \\
  ( a \circv ( b \circh e ) ) \circh ( c \circv d )
  &\equiv
  ( a \circh c ) \circv ( ( b \circh e ) \circh d )
  \equiv
  ( a \circh c ) \circv ( b \circh ( e \circh d ) ).
  \end{align*}
The first consequence is already in normal form: it gives edges joining association types 26 and 68, 
both labelled by the transposition of positions 2 and 3.
We have right-normed the associative product of $b, e, d$ in the second consequence; permuting the 
variables gives the following normal form:
  \[
  ( a \circv ( b \circh c ) ) \circh ( d \circv e ) \equiv ( a \circh d ) \circv ( b \circh ( c \circh e ) ).
  \]
The second consequence gives edges joining association types 30 and 68, with the first labelled 
by $(234)$ in terms of positions, or $(243)$ in terms of indeterminates.
Putting this together, we obtain this connected component of $G(5)$:
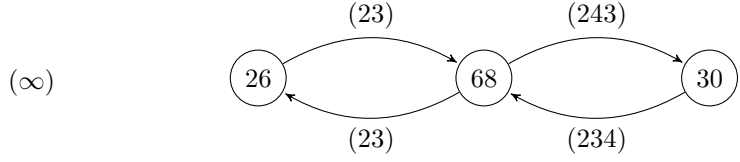
\begin{equation}
\label{degree5component}
\tag{$\infty$}
\begin{tabular}{c}
\begin{tikzpicture}
[
->,
>=stealth',
shorten >=1pt,
auto,
node distance=3cm,
main node/.style={circle,draw,font=\normalsize}
]
  \node[main node] (1) {26};
  \node[main node] (2) [right of=1] {68};
  \node[main node] (3) [right of=2] {30};
  \path[every node/.style={font=\normalsize}]
  (1) edge [bend left] node [above] {$(23)$} (2)
  (2) edge [bend left] node [below] {$(23)$} (1)
  (2) edge [bend left] node [above] {$(243)$} (3)
  (3) edge [bend left] node [below] {$(234)$} (2)
  ;
\end{tikzpicture}
\end{tabular}
\end{equation}
To illustrate the composition of edges, we travel from vertex 30 (with the identity permutation) to vertex 26.
The first and second edges represent these consequences:
  \begin{align*}
  ( a \circv ( b \circh c ) ) \circh ( d \circv e )
  &\equiv
  ( a \circh d ) \circv ( b \circh ( c \circh e ) ),
  \\
  ( a \circh b ) \circv ( c \circh ( d \circh e ) )
  &\equiv
  ( a \circv c ) \circh ( b \circv ( d \circh e ) ).
  \end{align*}
To compose the permutations, we substitute the indeterminates in the right side of the first consequence into 
both sides of the second, and obtain:
  \[
  ( a \circh d ) \circv ( b \circh ( c \circh e ) )
  \equiv
  ( a \circv b ) \circh ( d \circv ( c \circh e ) ).
  \]
Combining this with the first consequence produces the composition of the two consequences corresponding to 
two applications  of the interchange relation \eqref{interchange}:
  \[
  ( a \circv ( b \circh c ) ) \circh ( d \circv e )
  \equiv
  ( a \circv b ) \circh ( d \circv ( c \circh e ) ).
  \]
The right side of this result corresponds to the equation $(23)(234) = (34)$ obtained by composing the
permutations along the bottom arrows of \eqref{degree5component}.
\end{example}

The next definition expresses commutativity properties in double semigroups in terms of the directed graphs
$F(n)$ and $G(n)$ and the covering map $p\colon F(n) \to G(n)$.

\begin{definition} \label{definitionnontrivial}
A \textbf{commutativity property} in degree $n$ for double semigroups corresponds to a directed cycle $C$ 
in the quotient graph $G(n)$ for which the product of the permutations along the edges is not the identity.
In this case we call $C$ a \textbf{nontrivial cycle}. 
A nontrivial cycle is a closed path in $G(n)$ for which some  (and hence any) lifting to $F(n)$ is not closed.  
If a nontrivial cycle passes through association type $t$, then we may regard $t$ as the first and last vertex 
of the cycle, and obtain a commutativity property of the form $m \equiv m'$, 
where both monomials have association type $t$ but different permutations of the indeterminates.
\end{definition}

The directed graph $G(n)$ decomposes into a number of isolated vertices together with a number of 
(nontrivial connected) components.
If association types $t_i$ and $t_j$ are vertices in the same component, then any multilinear monomial 
$m_i$ with association type $t_i$ can be rewritten by a sequence of consequences of the interchange relation 
\eqref{interchange} as a suitably permuted monomial $m_j$ with association type $t_j$.
This is the unique lifting of a path in $G(n)$ to a path in $F(n)$.
It follows that the number of distinct monomials in $n$ variables in the free double semigroup on one generator 
is equal to the sum of the number of isolated vertices in $G(n)$ and the number of (nontrivial connected) 
components in $G(n)$.
Since we are concerned with directed cycles in $G(n)$, we may study the connected components independently.
We choose one of the association types $t_i$ in each component as the normal form; for computational convenience
we usually take the minimal association type in the total order of Definition \ref{definitiontotalorder}.
We then determine how every other association type $t_j$ in that component can be reduced to $t_i$ 
with some corresponding permutation of the indeterminates by following a directed path from $t_j$ to $t_i$.

Although we will not pursue this point of view, we mention that the directed graph $G(n)$ can be regarded as 
the generating set for a groupoid \cite{Brown1987}: the objects are the vertices and the morphisms are the 
identity maps together with all compositions of the directed edges.
This groupoid is not necessarily free: directed cycles express relations that must be satisfied by the generators.

\begin{remark}
The existence of nontrivial cycles amounts to nontrivial monodromy of the covering map $p\colon F(n) \to G(n)$.  
For a fixed vertex $t \in VG(n)$, the fundamental group $\pi(G(n),t)$ acts on the fiber $\mathcal{F} = p^{-1}(t)$,
which consists of all bijections from the leaves of $Y(t)$ to the indeterminates $X$ 
(all assignments of indeterminates to the positions in the association type).
The monodromy group is the image of the group action $\pi(G(n),t) \to \mathrm{Aut}(\mathcal{F})$.
This conceptual description of $G(n)$ does not require orderings or normal forms of any kind.  
The role of orderings and normal forms is merely to allow us to choose a section to the covering map 
$p\colon F(n) \to G(n)$, individually on vertices and on edges (and clearly this section is 
far from being a graph homomorphism).
These choices are very useful, indeed essential, for the computer algebra, but it must be emphasized that they
are secondary, not intrinsic to the statements of the results or to the definitions of the 
objects under study.

The topological interpretations of graphs go through the category of groupoids: a path in the directed graph 
$G(n)$ is just an arrow in the free groupoid on $G(n)$, the fundamental group $\pi(G(n),t)$ is just the vertex 
group in the free groupoid, and so on.  
The consequences of the interchange relation generate a groupoid (they do so globally: it is not necessary 
to treat the connected components individually).  
The meaning of the term `generates' in this context deserves further analysis.  
One option is to generate the free groupoid on $G(n)$.
The monodromy groups of the covering projection for the graphs coincide with the monodromy groups of the 
covering projection for groupoids generated by the graphs.

\begin{figure}[ht]
\[
\begin{array}{cc}
\begin{array}{c}
\begin{tikzpicture}[draw=black, x=\Size, y=\Size]
\node [Square] at ($(0,0)$) {$a$};
\node [Square] at ($(1,0)$) {$b$};
\node [Square] at ($(2.4,0)$) {$c$};
\node [Square] at ($(0,-1)$) {$d$};
\node [Square] at ($(1,-1)$) {$e$};
\node [Square] at ($(2.4,-1)$) {$f$};
\end{tikzpicture}
\\
\big( ( a \circv d ) \circh ( b \circv e ) \big) \circh ( c \circv f )
\end{array}
&\qquad
\begin{array}{c}
\begin{tikzpicture}[draw=black, x=\Size, y=\Size]
\node [Square] at ($(0,0)$) {$a$};
\node [Square] at ($(1.4,0)$) {$b$};
\node [Square] at ($(2.4,0)$) {$c$};
\node [Square] at ($(0,-1)$) {$d$};
\node [Square] at ($(1.4,-1)$) {$e$};
\node [Square] at ($(2.4,-1)$) {$f$};
\end{tikzpicture}
\\
( a \circv d ) \circh \big( ( b \circv e ) \circh ( c \circv f ) \big) 
\end{array}
\\
\\
\begin{array}{c}
\begin{tikzpicture}[draw=black, x=\Size, y=\Size]
\node [Square] at ($(0,0)$) {$a$};
\node [Square] at ($(1,0)$) {$b$};
\node [Square] at ($(2.2,0)$) {$c$};
\node [Square] at ($(0,-1.4)$) {$d$};
\node [Square] at ($(1,-1.4)$) {$e$};
\node [Square] at ($(2.2,-1.4)$) {$f$};
\end{tikzpicture}
\\
\big( ( a \circh b ) \circh c \big) \circv \big( ( d \circh e ) \circh f \big)
\end{array}
&\qquad
\begin{array}{c}
\begin{tikzpicture}[draw=black, x=\Size, y=\Size]
\node [Square] at ($(0,0)$) {$a$};
\node [Square] at ($(1.2,0)$) {$b$};
\node [Square] at ($(2.2,0)$) {$c$};
\node [Square] at ($(0,-1.4)$) {$d$};
\node [Square] at ($(1.2,-1.4)$) {$e$};
\node [Square] at ($(2.2,-1.4)$) {$f$};
\end{tikzpicture}
\\
\big( a \circh ( b \circh c ) \big) \circv \big( d \circh ( e \circh f ) \big)
\end{array}
\end{array}
\]
\vspace{-4mm}
\caption{Equivalent factorizations of the $2 \times 3$ array}
\label{2x3array}
\end{figure}

However, there are obvious relations to impose; for convenience we will use the term \emph{primitive move}
to mean the application of a consequence of the interchange relation.
Two different sequences of primitive moves deserve to be identified when (for example) doing one move 
followed by another move in a remote region of the association type: it is reasonable to consider this 
the same as doing the moves in the opposite order.  
Similarly, associativity accounts to some extent for different ways of expressing an arrow in terms of 
primitive moves: consider the $2 \times 3$ array of boxes, as illustrated in Figure \ref{2x3array}.
The arrow from the vertical-horizontal 
factorization to the horizontal-vertical factorization can be realized in different ways as a composition of 
primitive moves, depending on the way in which the triple composition is bracketed.
See the discussion of Figure \ref{degree6component} in the next section for further details on this issue.
The groupoid generated by the graph as understood in this sense is clearly more complicated than the free 
groupoid.  
However, imposing these relations (which need to be made precise) will not affect the monodromy 
groups, since paths declared to be equal induce the same permutation of the elements.
\end{remark}


\section{Every cycle is trivial in degree $n \le 8$} \label{section45678}

In this section we describe our computational verification that that there are no commutativity relations
in degree $n \le 8$.
We may regard a consequence $R$ of the interchange relation in degree $n$ as the equality of two geometric 
realizations.
Since the two outermost arguments of the original interchange relation \eqref{interchange} remain unchanged, 
we expect that the elements in the boxes on the boundary of the diagram realizing $R$ cannot be permuted.
This principle (which needs to be made precise) quickly rules out the possibility of nontrivial cycles 
for most association types of low degree.
Nonetheless, it is clearly much simpler to provide detailed examples to illustrate our computational methods
when the degree $n$ is small.

The results described in this section were obtained using computer algebra, especially the Maple package 
\texttt{GraphTheory}.
We wrote our own procedures to generate the association types and the normalized consequences of
the interchange relation.
We then used \texttt{GraphTheory} to decompose $G(n)$ into its connected components and 
to find a generating set for the cycles in each component.
We wrote our own procedures to decide if any cycles were nontrivial in the sense of 
Definition \ref{definitionnontrivial}.

In the following examples, we label the vertices of $G(n)$ by the index numbers of their association types: 
that is, we write $i$ instead of $t_i$.

\subsection*{Degree 5}

The graph $G(5)$ was introduced in Example \ref{exampledegree5} together with one of its (nontrivial 
connected) components.
Figure \ref{degree5table} gives the index numbers of the minimal vertices and the corresponding association types.
Figure \ref{degree5components} displays the 8 components; it is clear that all cycles are trivial.
Transposing $\circh$ and $\circv$ partitions the vertices into 10 orbits:
$8 / 63$,
$18 / 53$,
$25 / 73$,
$26 / 69$,
$27 / 68$,
$28 / 70$,
$30 / 76$,
$32 / 74$,
$34 / 78$,
$41 / 87$.
This action by the cyclic group of order 2 extends to every degree.

\begin{figure}[ht]
\begin{alignat*}{3}
 8&\colon \, a \circh ( b \circv c ) \circh ( d \circv e )  &\qquad
18&\colon \, a \circh ( ( b \circh c ) \circv ( d \circh e ) )  &\qquad
25&\colon \, ( a \circv b ) \circh ( c \circv d ) \circh e  \\
26&\colon \, ( a \circv b ) \circh ( c \circv ( d \circh e ) )  &\qquad
27&\colon \, ( a \circv b ) \circh ( c \circv d \circv e )  &\qquad
28&\colon \, ( a \circv b ) \circh ( ( c \circh d ) \circv e )  \\
30&\colon \, ( a \circv ( b \circh c ) ) \circh ( d \circv e )  &\qquad
32&\colon \, ( a \circv b \circv c ) \circh ( d \circv e )  &\qquad
34&\colon \, ( ( a \circh b ) \circv c ) \circh ( d \circv e )  \\
41&\colon \, ( ( a \circh b ) \circv ( c \circh d ) ) \circh e  &\qquad
53&\colon \, a \circv ( ( b \circv c ) \circh ( d \circv e ) )  &\qquad
63&\colon \, a \circv ( b \circh c ) \circv ( d \circh e )  \\
68&\colon \, ( a \circh b ) \circv ( c \circh d \circh e )  &\qquad
69&\colon \, ( a \circh b ) \circv ( c \circh ( d \circv e ) )  &\qquad
70&\colon \, ( a \circh b ) \circv ( ( c \circv d ) \circh e )  \\
73&\colon \, ( a \circh b ) \circv ( c \circh d ) \circv e  &\qquad
74&\colon \, ( a \circh b \circh c ) \circv ( d \circh e )  &\qquad
76&\colon \, ( a \circh ( b \circv c ) ) \circv ( d \circh e )  \\
78&\colon \, ( ( a \circv b ) \circh c ) \circv ( d \circh e )  &\qquad
87&\colon \, ( ( a \circv b ) \circh ( c \circv d ) ) \circv e
\end{alignat*}
\vspace{-7mm}
\caption{Association types of connected components in degree 5}
\label{degree5table}
\end{figure}

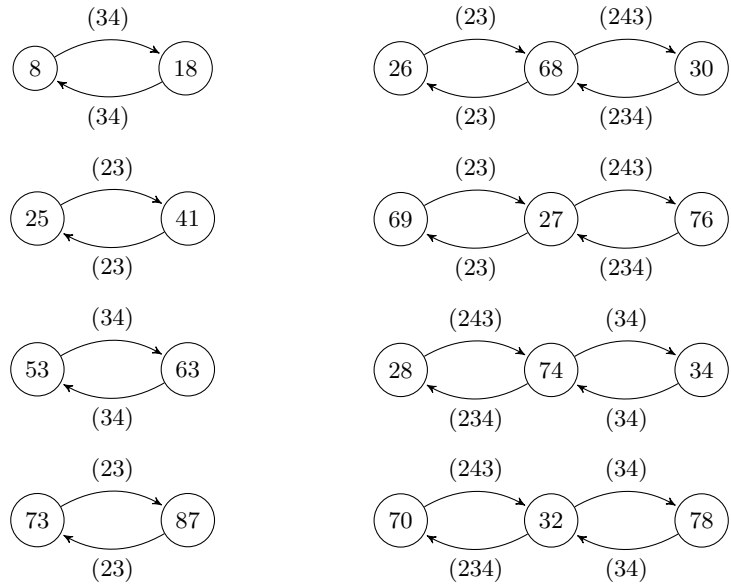
\begin{figure}[ht]
\begin{tabular}{cc}
\begin{tabular}{c}
\begin{tikzpicture}
[
->,
>=stealth',
shorten >=1pt,
auto,
node distance=2cm,
main node/.style={circle,draw,font=\small}
]
  \node[main node] (1) {8};
  \node[main node] (2) [right of=1] {18};
  \path[every node/.style={font=\small}]
  (1) edge [bend left] node [above] {$(34)$} (2)
  (2) edge [bend left] node [below] {$(34)$} (1)
  ;
\end{tikzpicture}
\\
\begin{tikzpicture}
[
->,
>=stealth',
shorten >=1pt,
auto,
node distance=2cm,
main node/.style={circle,draw,font=\small}
]
  \node[main node] (1) {25};
  \node[main node] (2) [right of=1] {41};
  \path[every node/.style={font=\small}]
  (1) edge [bend left] node [above] {$(23)$} (2)
  (2) edge [bend left] node [below] {$(23)$} (1)
  ;
\end{tikzpicture}
\\
\begin{tikzpicture}
[
->,
>=stealth',
shorten >=1pt,
auto,
node distance=2cm,
main node/.style={circle,draw,font=\small}
]
  \node[main node] (1) {53};
  \node[main node] (2) [right of=1] {63};
  \path[every node/.style={font=\small}]
  (1) edge [bend left] node [above] {$(34)$} (2)
  (2) edge [bend left] node [below] {$(34)$} (1)
  ;
\end{tikzpicture}
\\
\begin{tikzpicture}
[
->,
>=stealth',
shorten >=1pt,
auto,
node distance=2cm,
main node/.style={circle,draw,font=\small}
]
  \node[main node] (1) {73};
  \node[main node] (2) [right of=1] {87};
  \path[every node/.style={font=\small}]
  (1) edge [bend left] node [above] {$(23)$} (2)
  (2) edge [bend left] node [below] {$(23)$} (1)
  ;
\end{tikzpicture}
\end{tabular}
&\qquad\qquad
\begin{tabular}{c}
\begin{tikzpicture}
[
->,
>=stealth',
shorten >=1pt,
auto,
node distance=2cm,
main node/.style={circle,draw,font=\small}
]
  \node[main node] (1) {26};
  \node[main node] (2) [right of=1] {68};
  \node[main node] (3) [right of=2] {30};
  \path[every node/.style={font=\small}]
  (1) edge [bend left] node [above] {$(23)$} (2)
  (2) edge [bend left] node [below] {$(23)$} (1)
  (2) edge [bend left] node [above] {$(243)$} (3)
  (3) edge [bend left] node [below] {$(234)$} (2)
  ;
\end{tikzpicture}
\\
\begin{tikzpicture}
[
->,
>=stealth',
shorten >=1pt,
auto,
node distance=2cm,
main node/.style={circle,draw,font=\small}
]
  \node[main node] (1) {69};
  \node[main node] (2) [right of=1] {27};
  \node[main node] (3) [right of=2] {76};
  \path[every node/.style={font=\small}]
  (1) edge [bend left] node [above] {$(23)$} (2)
  (2) edge [bend left] node [below] {$(23)$} (1)
  (2) edge [bend left] node [above] {$(243)$} (3)
  (3) edge [bend left] node [below] {$(234)$} (2)
  ;
\end{tikzpicture}
\\
\begin{tikzpicture}
[
->,
>=stealth',
shorten >=1pt,
auto,
node distance=2cm,
main node/.style={circle,draw,font=\small}
]
  \node[main node] (1) {28};
  \node[main node] (2) [right of=1] {74};
  \node[main node] (3) [right of=2] {34};
  \path[every node/.style={font=\small}]
  (1) edge [bend left] node [above] {$(243)$} (2)
  (2) edge [bend left] node [below] {$(234)$} (1)
  (2) edge [bend left] node [above] {$(34)$} (3)
  (3) edge [bend left] node [below] {$(34)$} (2)
  ;
\end{tikzpicture}
\\
\begin{tikzpicture}
[
->,
>=stealth',
shorten >=1pt,
auto,
node distance=2cm,
main node/.style={circle,draw,font=\small}
]
  \node[main node] (1) {70};
  \node[main node] (2) [right of=1] {32};
  \node[main node] (3) [right of=2] {78};
  \path[every node/.style={font=\small}]
  (1) edge [bend left] node [above] {$(243)$} (2)
  (2) edge [bend left] node [below] {$(234)$} (1)
  (2) edge [bend left] node [above] {$(34)$} (3)
  (3) edge [bend left] node [below] {$(34)$} (2)
  ;
\end{tikzpicture}
\end{tabular}
\end{tabular}
\vspace{-3mm}
\caption{Connected components in degree 5}
\label{degree5components}
\end{figure}

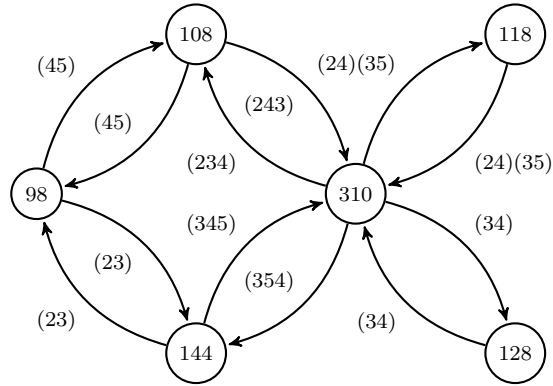
\begin{figure}[ht]
\begin{tikzpicture}
[
->,
>=stealth',
shorten >=1pt,
auto,
node distance=3cm,
thick,
main node/.style={circle,draw,font=\footnotesize}
]
  \node[main node] (1) {98};
  \node[main node] (2) [above right of=1] {108};
  \node[main node] (3) [below right of=1] {144};
  \node[main node] (4) [below right of=2] {310};
  \node[main node] (5) [above right of=4] {118};
  \node[main node] (6) [below right of=4] {128};
  \path[every node/.style={font=\footnotesize}]
    (1) edge [bend left] node {$(45)$} (2)
        edge [bend left] node [below left] {$(23)$} (3)
    (2) edge [bend left] node [above left] {$(45)$} (1)
        edge [bend left] node [below left] {$(243)$} (4)
    (3) edge [bend left] node {$(345)$} (4)
        edge [bend left] node {$(23)$} (1)
    (4) edge [bend left] node {$(234)$} (2)
        edge [bend left] node [above left] {$(354)$} (3)
        edge [bend left] node {$(24)(35)$} (5)
        edge [bend left] node {$(34)$} (6)
    (5) edge [bend left] node {$(24)(35)$} (4)
    (6) edge [bend left] node {$(34)$} (4)
        ;
\end{tikzpicture}
\vspace{-3mm}
\caption{One connected component of $G(6)$}
\label{degree6component}
\end{figure}

\subsection*{Degree 6}

The graph $G(6)$ has 394 vertices and 98 pairs of edges with 254 isolated vertices and 
44 nontrivial connected components: respectively 14, 20, 4, 6 components with 2, 3, 4, 6 vertices.
Two components of size 6 have cycles; they are related by transposing $\circh$ and $\circv$.
The vertices and edges of the component in Figure \ref{degree6component} correspond to the association 
types and consequences of the interchange relation given Figure \ref{degree6info}.
In degree 6 there are cycles, but they are all trivial;
for example, the product of the permutations along the edges of the cycle $[98, 108, 310, 144, 98]$ 
is the identity permutation: $(23)(354)(243)(45) = ()$.
This cycle expresses two different ways of going from vertex 98 to vertex 310; equivalently,
in the $2 \times 3$ array of Figure \ref{2x3array}, from outer-black/inner-white to outer-white/inner-black.
These two ways differ only by an application of associativity, and hence in the groupoid generated 
by the graph $G(6)$, these two paths would be equal.

\begin{figure}[ht]
\begin{alignat*}{2}
 98\colon &( a \circv b ) \circh ( c \circv d ) \circh ( e \circv f )
&\qquad\qquad
108\colon &( a \circv b ) \circh ( ( c \circh d ) \circv ( e \circh f ) )
\\
118\colon &( a \circv ( b \circh c ) ) \circh ( ( d \circh e ) \circv f )
&\qquad\qquad
128\colon &( ( a \circh b ) \circv c ) \circh ( d \circv ( e \circh f ) )
\\
144\colon &( ( a \circh b ) \circv ( c \circh d ) ) \circh ( e \circv f )
&\qquad\qquad
310\colon &( a \circh b \circh c ) \circv ( d \circh e \circh f )
\end{alignat*}
\begin{align*}
( a \circv b ) \circh ( c \circv d ) \circh ( e \circv f ) &\equiv
( a \circv b ) \circh ( ( c \circh e ) \circv ( d \circh f ) )
\\
( a \circv b ) \circh ( c \circv d ) \circh ( e \circv f ) &\equiv
( ( a \circh c ) \circv ( b \circh d ) ) \circh ( e \circv f )
\\
( a \circv b ) \circh ( ( c \circh d ) \circv ( e \circh f ) ) &\equiv
( a \circh c \circh d ) \circv ( b \circh e \circh f )
\\
( a \circv ( b \circh c ) ) \circh ( ( d \circh e ) \circv f ) &\equiv
( a \circh d \circh e ) \circv ( b \circh c \circh f )
\\
( ( a \circh b ) \circv c ) \circh ( d \circv ( e \circh f ) ) &\equiv
( a \circh b \circh d ) \circv ( c \circh e \circh f )
\\
( ( a \circh b ) \circv ( c \circh d ) ) \circh ( e \circv f ) &\equiv
( a \circh b \circh e ) \circv ( c \circh d \circh f )
\end{align*}
\vspace{-7mm}
\caption{Vertices and edges of the component in Figure \ref{degree6component}}
\label{degree6info}
\end{figure}

\begin{definition}
The \textbf{circuit rank} (or cyclomatic number) of a graph $G$ is the minimum number $r$ of edges
which must be removed to eliminate all of the cycles, thus resulting in a forest (a disjoint union of trees).
\end{definition}

\begin{lemma}
If $G$ has $v$ vertices, $e$ edges, and $c$ components (including the isolated vertices), then $r = e - v + c$.
In particular, if $G$ is connected then $r = e - v + 1$.
Every graph has a cycle basis: a set of cycles which generate all of the cycles by means of 
repeated symmetric differences.
Every cycle basis for $G$ has the same number of cycles, and this is called the circuit rank of $G$.
\end{lemma}

\begin{proof}
See Berge \cite{Berge2001}, starting on page 27.
\end{proof}

\subsection*{Degree 7}

The graph $G(7)$ has 1806 vertices and 688 pairs of edges, with 948 isolated vertices 
and 210 (nontrivial connected) components.
Our computations found the following number of components for each size (number of vertices):
  \[
  \begin{tabular}{lrrrrrrrrrrr}
  vertices &\; 1 &\; 2 &\; 3 &\; 4 &\; 5 &\; 6 &\; 10 &\; 11 &\; 12 &\; 13 \\
  components &\; 948 &\; 52 &\; 84 &\; 24 &\; 4 &\; 24 &\; 12 &\; 2 &\; 4 &\; 4
  \end{tabular}
  \]
The command \texttt{CycleBasis} from the Maple package \texttt{GraphTheory} showed that 22 components
have cycles: for circuit ranks 1, 2, 3 there are 12, 2, 8 components respectively.
Our computations also showed that all of the cycles are trivial.
Of the 8 components with rank 3, four have 10 vertices and four have 13 vertices.

\begin{figure}[ht]
\begin{tikzpicture}
[
->,
>=stealth',
shorten >=1pt,
auto,
node distance=2.0cm,
thick,
main node/.style={circle,draw,font=\footnotesize}
]
  \node[main node] (421) {421};
  \node[main node] (463) [left of=421] {463};
  \node[main node] (470) [right of=421] {470};
  \node[main node] (591) [below=0.6cm of 421] {591};
  \node[main node] (1389) [below=0.6cm of 463] {1389};
  \node[main node] (1460) [below=0.6cm of 470] {1460};
  \node[main node] (541) [left of=1389] {541};
  \node[main node] (505) [right of=1460] {505};
  \node[main node] (503) [below=0.6cm of 1389] {503};
  \node[main node] (543) [below=0.6cm of 1460] {543};
  \node[main node] (1484) [below=0.6cm of 503] {1484};
  \node[main node] (1391) [below=0.6cm of 543] {1391};
  \node[main node] (465) [right of=1484] {465};
  \path[every node/.style={font=\footnotesize}]
  ( 421) edge [above] node {$(45)$} ( 463)
  ( 421) edge node {$(465)$} ( 470)
  ( 421) edge node {$(23)$} ( 591)
  ( 463) edge node {$(243)$} (1389)
  ( 465) edge node {$(243)$} (1391)
  ( 465) edge [above] node {$(2543)$} (1484)
  ( 470) edge node {$(2543)$} (1460)
  ( 503) edge [right] node {$(24)(35)$} (1389)
  ( 503) edge node {$(25364)$} (1484)
  ( 505) edge [above] node {$(25364)$} (1460)
  ( 541) edge node {$(34)$} (1389)
  ( 543) edge node {$(34)$} (1391)
  ( 543) edge [right] node {$(354)$} (1460)
  ( 591) edge [above] node {$(345)$} (1389)
  ( 591) edge [above] node {$(35)(46)$} (1460)
  ;
\end{tikzpicture}
\caption{One connected component of $G(7)$}
\label{degree7figure}
\end{figure}
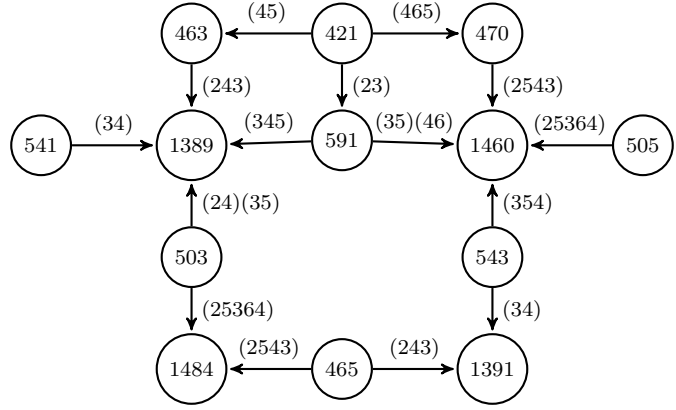

\begin{example} \label{degree7example}
Figure \ref{degree7figure} displays a component of $G(7)$ with 13 vertices, 15 edges, and circuit rank 3.
The shape is rectangular, since we show only one edge from each pair with its permutation;
the reverse edge has the inverse permutation.
The vertices correspond to the following association types:
  \begin{alignat*}{2}
   421\colon  &( a \circv b ) \circh ( ( c \circv d ) \circh ( e \circv ( f \circv g ) ) ) &\qquad\quad
   463\colon  &( a \circv b ) \circh ( ( c \circh d ) \circv ( e \circh ( f \circv g ) ) ) \\ 
   465\colon  &( a \circv b ) \circh ( ( c \circh d ) \circv ( e \circv ( f \circh g ) ) ) &\qquad\quad
   470\colon  &( a \circv b ) \circh ( ( c \circh ( d \circv e ) ) \circv ( f \circh g ) ) \\ 
   503\colon  &( a \circv ( b \circh c ) ) \circh ( ( d \circh e ) \circv ( f \circv g ) ) &\qquad\quad
   505\colon  &( a \circv ( b \circh c ) ) \circh ( ( d \circh ( e \circv f ) ) \circv g ) \\ 
   541\colon  &( ( a \circh b ) \circv c ) \circh ( d \circv ( e \circh ( f \circv g ) ) ) &\qquad\quad
   543\colon  &( ( a \circh b ) \circv c ) \circh ( d \circv ( e \circv ( f \circh g ) ) ) \\ 
   591\colon  &( ( a \circh b ) \circv ( c \circh d ) ) \circh ( e \circv ( f \circv g ) ) &\qquad\quad
  1389\colon  &( a \circh ( b \circh c ) ) \circv ( d \circh ( e \circh ( f \circv g ) ) ) \\ 
  1391\colon  &( a \circh ( b \circh c ) ) \circv ( d \circh ( e \circv ( f \circh g ) ) ) &\qquad\quad
  1460\colon  &( a \circh ( b \circh ( c \circv d ) ) ) \circv ( e \circh ( f \circh g ) ) \\ 
  1484\colon  &( a \circh ( ( b \circh c ) \circv d ) ) \circv ( e \circh ( f \circh g ) )
  \end{alignat*}
The edges correspond to these consequences of the interchange relation:
\begin{align*}
( a \circv b ) \circh ( c \circv d ) \circh ( e \circv f \circv g ) &\equiv
( a \circv b ) \circh ( ( c \circh e ) \circv ( d \circh ( f \circv g ) ) ) \\
( a \circv b ) \circh ( c \circv d ) \circh ( e \circv f \circv g ) &\equiv
( a \circv b ) \circh ( ( c \circh ( e \circv f ) ) \circv ( d \circh g ) ) \\
( a \circv b ) \circh ( c \circv d ) \circh ( e \circv f \circv g ) &\equiv
( ( a \circh c ) \circv ( b \circh d ) ) \circh ( e \circv f \circv g ) \\
( a \circv b ) \circh ( ( c \circh d ) \circv ( e \circh ( f \circv g ) ) ) &\equiv
( a \circh c \circh d ) \circv ( b \circh e \circh ( f \circv g ) ) \\
( a \circv b ) \circh ( ( c \circh d ) \circv e \circv ( f \circh g ) ) &\equiv
( a \circh c \circh d ) \circv ( b \circh ( e \circv ( f \circh g ) ) ) \\
( a \circv b ) \circh ( ( c \circh d ) \circv e \circv ( f \circh g ) ) &\equiv
( a \circh ( ( c \circh d ) \circv e ) ) \circv ( b \circh f \circh g ) \\
( a \circv b ) \circh ( ( c \circh ( d \circv e ) ) \circv ( f \circh g ) ) &\equiv
( a \circh c \circh ( d \circv e ) ) \circv ( b \circh f \circh g ) \\
( a \circv ( b \circh c ) ) \circh ( ( d \circh e ) \circv f \circv g ) &\equiv
( a \circh d \circh e ) \circv ( b \circh c \circh ( f \circv g ) ) \\
( a \circv ( b \circh c ) ) \circh ( ( d \circh e ) \circv f \circv g ) &\equiv
( a \circh ( ( d \circh e ) \circv f ) ) \circv ( b \circh c \circh g ) \\
( a \circv ( b \circh c ) ) \circh ( ( d \circh ( e \circv f ) ) \circv g ) &\equiv
( a \circh d \circh ( e \circv f ) ) \circv ( b \circh c \circh g ) \\
( ( a \circh b ) \circv c ) \circh ( d \circv ( e \circh ( f \circv g ) ) ) &\equiv
( a \circh b \circh d ) \circv ( c \circh e \circh ( f \circv g ) ) \\
( ( a \circh b ) \circv c ) \circh ( d \circv e \circv ( f \circh g ) ) &\equiv
( a \circh b \circh d ) \circv ( c \circh ( e \circv ( f \circh g ) ) ) \\
( ( a \circh b ) \circv c ) \circh ( d \circv e \circv ( f \circh g ) ) &\equiv
( a \circh b \circh ( d \circv e ) ) \circv ( c \circh f \circh g ) \\
( ( a \circh b ) \circv ( c \circh d ) ) \circh ( e \circv f \circv g ) &\equiv
( a \circh b \circh e ) \circv ( c \circh d \circh ( f \circv g ) ) \\
( ( a \circh b ) \circv ( c \circh d ) ) \circh ( e \circv f \circv g ) &\equiv
( a \circh b \circh ( e \circv f ) ) \circv ( c \circh d \circh g )
\end{align*}
One choice of cycle basis consists of the following three cycles:
  \begin{align*}
  &
  [ 421, 463, 1389, 591, 421 ], \qquad\qquad
  [ 421, 470, 1460, 591, 421 ],
  \\
  &
  [ 421, 463, 1389, 503, 1484, 465, 1391, 543, 1460, 470, 421 ].
  \end{align*}
The compositions of the permutations along the edges of these cycles are as follows;
in every case we obtain the identity permutation:
  \begin{align*}
  &
  (23)(354)(243)(45) = (\,), \qquad\qquad
  (23)(35)(46)(2543)(465) = (\,),
  \\
  &
  (456)(2345)(354)(34)(243)(2345)(25364)(24)(35)(243)(45) = (\,).
  \end{align*}
Since all other cycles can be obtained by repeated symmetric differences from the cycle basis,
it follows that all cycles in this component are trivial.
\end{example}

\subsection*{Degree 8}

The graph $G(8)$ has 8558 vertices and 4482 pairs of edges with 3618 isolated vertices 
and 931 (nontrivial connected) components.
The following table gives the number $c$ of components for each number $v \ge 2$ of vertices:
  \[
  \begin{array}{rrrrrrrrrr}
  v &\; 2 &\; 3 &\; 4 &\; 5 &\; 6 &\; 7 &\; 10 &\; 11 &\; 12 \\
  c &\; 204 &\; 352 &\; 112 &\; 28 &\; 88 &\; 4 &\; 48 &\; 10 &\; 16 \\ \midrule
  v &\; 13 &\; 15 &\; 16 &\; 20 &\; 24 &\; 30 &\; 33 &\; 36 &\; 38 \\
  c &\; 16 &\; 12 &\; 2 &\; 6 &\; 12 &\; 5 &\; 4 &\; 8 &\; 4
  \end{array}
  \]
There are 131 components with nonzero circuit rank; the following table gives 
the number $c$ of components for each circuit rank $r \ge 1$.
Only one component has rank 13, so its set 
of vertices must be invariant under transposing $\circh$ and $\circv$:
  \[
  \begin{array}{rrrrrrrrrrr}
  r &\;  1 &\;  2 &\;  3 &\;  4 &\; 6 &\; 9 &\; 10 &\; 11 &\; 13 &\; 14 \\
  c &\; 44 &\; 14 &\; 34 &\; 12 &\; 8 &\; 8 &\;  4 &\;  2 &\;  1 &\;  4
  \end{array}
  \]
We found that all of the cycles in degree 8 are trivial.

\begin{proposition} \label{degree8nothing}
For $n \le 8$, every cycle in $G(n)$ is trivial in the sense of Definition \ref{definitionnontrivial}:
the product of the permutations along the edges is the identity.
Hence there are no commutativity properties for double semigroups in degree $n \le 8$.
\end{proposition}


\section{Existence of nontrivial cycles in degree 9} \label{sectiondegree9}

The directed graph $G(9)$ has 41586 vertices and 28004 pairs of edges with 14058 isolated vertices and 
4001 (nontrivial connected) components.
We sort the components first by number of vertices, and then by minimal vertex using the total order on 
association types from 
Definition \ref{definitiontotalorder}.
We find that 3380 components are trees; the other 621 components have 23 different nonzero circuit ranks.
The following table gives the number $c$ of components with circuit rank $r$ (including $r = 0$):
  \[
  \begin{array}{rrrrrrrrrrrrr}
  r &\; 0 &\; 1 &\; 2 &\; 3 &\; 4 &\; 6 &\; 9 &\; 10 &\; 11 &\; 13 &\; 14 &\; 15 \\
  c &\; 3380 &\; 162 &\; 66 &\; 132 &\; 66 &\; 32 &\; 32 &\; 32 &\; 8 &\; 4 &\; 16 &\; 4
  \\ \midrule
  r &\; 19 &\; 23 &\; 26 &\; 30 &\; 31 &\; 33 &\; 38 &\; 48 &\; 53 &\; 54 &\; 55 &\; 75 \\
  c &\; 16 &\;  4 &\; 4 &\; 8 &\; 8 &\; 4 &\; 4 &\; 4 &\; 4 &\; 2 &\; 8 &\; 1
  \end{array}
  \]
Recall that the circuit rank $r$ (the size of any cycle basis) is an invariant of the component.
It is the dimension of the cycle space: the vector space over $\mathbb{F}_2$ 
obtained from any cycle basis by iteration of symmetric difference (addition modulo 2).
But when we compute the product of the permutations along each cycle in a cycle basis, we can obtain different 
numbers of 
nontrivial cycles for different cycle bases.
As a canonical choice of cycle basis for each component, we take the basis computed by the Maple command 
\texttt{CycleBasis}, and convert it to an $r \times e$ matrix over $\mathbb{F}_2$ where $e$ is the number
of edges in the component; the $(i,j)$ entry is 1 if and only if cycle $i$ contains edge $j$.
We then compute the row canonical form (RCF) of the matrix over $\mathbb{F}_2$, and use the cycles represented 
by the rows of the RCF as our cycle basis.
Table \ref{nontrivial9} gives the 16 components which have nontrivial cycles in this cycle basis:
\begin{itemize}
\item[--]
column 1 is the index number of the component in the sorted list
\item[--]
column 2 is the number $v$ of vertices,
column 3 is the number $e$ of edges
\item[--]
column 4 is the circuit rank $r$
\item[--]
column 5 is the number $k$ of nontrivial cycles in the standard cycle basis
\item[--]
column 6 is the index number of the minimal vertex $v_{\min}$
\item[--]
column 7 is the corresponding minimal association type $t_{\min}$
\end{itemize}
Table \ref{nontrivial9} is a complete list, since if a component has a nontrivial cycle, 
then every cycle basis has at least one nontrivial cycle.
(Observe that if every basis cycle is trivial, then it is a consequence of our convention
for labelling edges with permutations that every linear combination over $\mathbb{F}_2$ is also trivial.) \,
Table \ref{nontrivial9} consists of three types of sizes 4, 4, 8 corresponding to circuit ranks 38, 53, 30 
respectively.
We choose a representative vertex in each component whose geometric realization has a particularly symmetric 
and rectangular form when the monomial is expressed as a diagram of boxes using the interpretation of 
the operations as horizontal and vertical compositions.
In each type, the geometric realizations of the components form a single orbit for the action of the
dihedral group $D_8$ (the symmetry group of the square); see the last column in Tables \ref{table9.1}, 
\ref{table9.2} and \ref{table9.3}.

\begin{table}[ht]
\[
\begin{array}{ccccccl}
\text{index} & v & e & r & k & v_{\min} & t_{\min} \\ \midrule
3981 & 107 & 144 & 38 &  2 &  9007 & 
( a \circv b ) \circh ( c \circv d ) \circh ( e \circv ( ( f \circv g \circv h ) \circh i ) )
\\
3982 & 107 & 144 & 38 &  2 &  9040 & 
( a \circv b ) \circh ( c \circv d ) \circh ( ( ( e \circv ( f \circh g ) ) \circh h ) \circv i )
\\
3983 & 107 & 144 & 38 &  2 & 11020 & 
( a \circv b \circv c ) \circh ( d \circv e \circv ( f \circh g \circh h ) \circv i )
\\
3984 & 107 & 144 & 38 &  2 & 11971 & 
( a \circv b \circv c \circv d ) \circh ( e \circv ( f \circh g \circh h ) \circv i )
\\
\midrule
3989 & 113 & 165 & 53 & 10 &  9017 & 
( a \circv b ) \circh ( c \circv d ) \circh ( e \circv ( f \circh g \circh h ) \circv i )
\\
3990 & 113 & 165 & 53 &  4 & 10898 & 
( a \circv b \circv c ) \circh ( d \circv e \circv f \circv g ) \circh ( h \circv i )
\\
3991 & 113 & 165 & 53 &  7 & 10975 &
( a \circv b \circv c ) \circh ( d \circv ( e \circh ( ( f \circh g \circh h ) \circv i ) ) )
\\
3992 & 113 & 165 & 53 &  5 & 11065 & 
( a \circv b \circv c ) \circh ( ( d \circh e ) \circv ( f \circh g \circh h ) \circv i )
\\
\midrule
3994 & 120 & 149 & 30 &  3 &  9481 & 
( a \circv b ) \circh ( c \circv ( d \circh e \circh ( ( f \circh g \circh h ) \circv i ) ) )
\\
3995 & 120 & 149 & 30 &  2 &  9526 & 
( a \circv b ) \circh ( c \circv ( d \circh ( e \circv ( f \circh g \circh h ) \circv i ) ) )
\\
3996 & 120 & 149 & 30 &  2 &  9920 & 
( a \circv b ) \circh ( ( c \circh d ) \circv e \circv ( f \circh g \circh h ) \circv i )
\\
3997 & 120 & 149 & 30 &  2 &  9965 & 
( a \circv b ) \circh ( ( c \circh d \circh e ) \circv ( f \circh g \circh h ) \circv i )
\\
3998 & 120 & 149 & 30 &  2 & 10965 & 
( a \circv b \circv c ) \circh ( d \circv ( e \circh ( f \circv g \circv h ) \circh i ) )
\\
3999 & 120 & 149 & 30 &  2 & 11116 & 
( a \circv b \circv c ) \circh ( ( d \circh ( e \circv ( f \circh g ) ) \circh h ) \circv i )
\\
4000 & 120 & 149 & 30 &  2 & 11961 & 
( a \circv b \circv c \circv d ) \circh ( e \circv ( ( f \circv g \circv h ) \circh i ) )
\\
4001 & 120 & 149 & 30 &  2 & 11994 & 
( a \circv b \circv c \circv d ) \circh ( ( ( e \circv ( f \circh g ) ) \circh h ) \circv i )
\\
\midrule
\end{array}
\]
\caption{Representative vertices of the 16 components for $n = 9$ which have nontrivial cycles,
partitioned according to circuit rank.}
\label{nontrivial9}
\end{table}

\subsection*{Type 1: components 3981 to 3984}

We represent these components by the vertices in Table \ref{table9.1}, expressed as association types and 
geometric realizations (diagrams); their index numbers in the total order on vertices are 
31182, 34464, 11028, 13963 respectively.
In the diagrams we use row vectors to denote horizontal composition ($\circh$)
and column vectors to denote vertical composition ($\circv$).
Each diagram has a horizontal or vertical symmetry, 
so there are four components in this orbit for the action of $D_8$.

\begin{table}[ht]
\begin{tabular}{ccc}
index &\; association type &\; diagram
\\ 
\\
3981 
&\;
$( a \circh b \circh c ) \circv ( d \circh ( e \circv f ) \circh ( g \circv h ) \circh i )$
&\;
$
\begin{bmatrix}
\begin{bmatrix} a & b & c \end{bmatrix}
\\
\\
\begin{bmatrix}
d 
&
\begin{bmatrix} e \\ f \end{bmatrix}
&
\begin{bmatrix} g \\ h \end{bmatrix}
&
i
\end{bmatrix}
\end{bmatrix}
$
\\
\\
3982 
&\;
$( a \circh ( b \circv c ) \circh ( d \circv e ) \circh f ) \circv ( g \circh h \circh i )$
&\;
$
\begin{bmatrix}
\begin{bmatrix}
a &
\begin{bmatrix} b \\ c \end{bmatrix}
&
\begin{bmatrix} d \\ e \end{bmatrix}
&
f
\end{bmatrix}
\\
\\
\begin{bmatrix} g & h & i \end{bmatrix}
\end{bmatrix}
$
\\
\\
3983 
&\;
$( a \circv b \circv c ) \circh ( d \circv ( e \circh f ) \circv ( g \circh h ) \circv i )$
&\;
$
\begin{bmatrix}
\begin{bmatrix} a \\ b \\ c \end{bmatrix}
&
\begin{bmatrix}
d \\
\begin{bmatrix} e & f \end{bmatrix} \\
\begin{bmatrix} g & h \end{bmatrix} \\
i
\end{bmatrix}
\end{bmatrix}
$
\\
\\
3984 
&\;
$( a \circv ( b \circh c ) \circv ( d \circh e ) \circv f ) \circh ( g \circv h \circv i )$
&\;
$
\begin{bmatrix}
\begin{bmatrix}
a \\
\begin{bmatrix} b & c \end{bmatrix} \\
\begin{bmatrix} d & e \end{bmatrix} \\
f
\end{bmatrix}
&
\begin{bmatrix} g \\ h \\ i \end{bmatrix}
\end{bmatrix}
$
\end{tabular}
\bigskip
\caption{Representative vertices for components of type 1}
\label{table9.1}
\end{table}

\begin{theorem} \label{theorem3981}
For component 3981, the nontrivial cycles in the cycle basis produce the following commutativity property
which transposes the indeterminates $\{ e, g \}$.
Every other association type in component 3981 produces a corresponding property;
since the association types in the same component represent the same equivalence class of monomials in 
the free double semigroup, starting from another association type amounts to choosing a different representative 
of the same equivalence class:
\begin{align*}
( a \circh b \circh c ) \circv ( d \circh ( e \circv f ) \circh ( g \circv h ) \circh i )
&\equiv
( a \circh b \circh c ) \circv ( d \circh ( g \circv f ) \circh ( e \circv h ) \circh i )
\\
\begin{bmatrix}
\begin{bmatrix} a & b & c \end{bmatrix}
\\
\begin{bmatrix}
d &
\begin{bmatrix} e \\ f \end{bmatrix}
&
\begin{bmatrix} g \\ h \end{bmatrix}
&
i
\end{bmatrix}
\end{bmatrix}
&\equiv
\begin{bmatrix}
\begin{bmatrix} a & b & c \end{bmatrix}
\\
\begin{bmatrix}
d &
\begin{bmatrix} g \\ f \end{bmatrix}
&
\begin{bmatrix} e \\ h \end{bmatrix}
&
i
\end{bmatrix}
\end{bmatrix}
\end{align*}
Cyclic rotation produces bijections between the four components of type 1 and hence 
there are corresponding identities in all components.
\end{theorem}

\begin{proof}
We use algebraic notation without geometric diagrams.
Applying the interchange relation to the factors $w, x, y, z$ means either
replacing $(w \circv x) \circh (y \circv z)$ by $(w \circh y) \circv (x \circh z)$ or the reverse;
this will be clear from the context.
To reduce the number of steps, we omit parentheses when associativity implies that this can be done 
without ambiguity.
At each step it suffices to specify the factors $w, x, y, z$ to which we apply the interchange relation, 
together with the result.
Figure \ref{proof3981} starts with the monomial on the left side of the relation in Theorem \ref{theorem3981},
and ends with the monomial on the right side.
\end{proof}

  \begin{figure}[ht]
  \begin{alignat*}{2}
  &
  \text{factors $w, x, y, z$}
  &\qquad\qquad
  &
  \text{result of interchange}
  \\
  &
  &\qquad\qquad
  &
  ( a \circh b \circh c ) \circv ( d \circh ( e \circv f ) \circh ( g \circv h ) \circh i )
  \\
  &
  a \circh b, \; c, \; d \circh ( e \circv f ) \circh ( g \circv h ), \; i
  &\qquad\qquad
  &
  ( ( a \circh b ) \circv ( d \circh ( e \circv f ) \circh ( g \circv h ) ) ) \circh ( c \circv i )
  \\
  &
  a, \; b, \; d \circh ( e \circv f ), \; g \circv h
  &\qquad\qquad
  &
  ( a \circv ( d \circh ( e \circv f ) ) ) \circh ( b \circv g \circv h ) \circh ( c \circv i )
  \\
  &
  a, \; d \circh ( e \circv f ), \; b \circv g, \; h
  &\qquad\qquad
  &
  ( ( a \circh ( b \circv g ) ) \circv ( d \circh ( e \circv f ) \circh h ) ) \circh ( c \circv i )
  \\
  &
  a \circh ( b \circv g ), \; d \circh ( e \circv f ) \circh h, \; c, \; i
  &\qquad\qquad
  &
  ( a \circh ( b \circv g ) \circh c ) \circv ( d \circh ( e \circv f ) \circh h \circh i )
  \\
  &
  a \circh ( b \circv g ) , \; c, \; d \circh ( e \circv f ), \; h \circh i
  &\qquad\qquad
  &
  ( ( a \circh ( b \circv g ) ) \circv ( d \circh ( e \circv f ) ) ) \circh ( c \circv ( h \circh i ) )
  \\
  &
  a, \; b \circv g, \; d, \; e \circv f
  &\qquad\qquad
  &
  ( a \circv d ) \circh ( ( b \circv g \circv e \circv f ) \circh ( c \circv ( h \circh i ) ) )
  \\
  &
  a, \; d, \; b \circv g \circv e, \; f
  &\qquad\qquad
  &
  ( ( a \circh ( b \circv g \circv e ) ) \circv ( d \circh f ) ) \circh ( c \circv ( h \circh i ) )
  \\
  &
  a \circh ( b \circv g \circv e ), \; d \circh f, \; c, \; h \circh i
  &\qquad\qquad
  &
  ( a \circh ( b \circv g \circv e ) \circh c ) \circv ( d \circh f \circh h \circh i )
  \\
  &
  a \circh ( b \circv g \circv e ), \; c, \; d \circh f \circh h, \; i
  &\qquad\qquad
  &
  ( ( a \circh ( b \circv g \circv e ) ) \circv ( d \circh f \circh h ) ) \circh ( c \circv i )
  \\
  &
  a, \; b \circv g \circv e, \; d \circh f, \; h
  &\qquad\qquad
  &
  ( a \circv ( d \circh f ) ) \circh ( b \circv g \circv e \circv h ) \circh ( c \circv i )
  \\
  &
  a, \; d \circh f, \; b \circv g, \; e \circv h
  &\qquad\qquad
  &
  ( ( a \circh ( b \circv g ) ) \circv ( d \circh f \circh ( e \circv h ) ) ) \circh ( c \circv i )
  \\
  &
  a \circh ( b \circv g ), \; d \circh f \circh ( e \circv h ), \; c, \; i
  &\qquad\qquad
  &
  ( a \circh ( b \circv g ) \circh c ) \circv ( d \circh f \circh ( e \circv h ) \circh i )
  \\
  &
  a \circh ( b \circv g ), \; c, \; d \circh f, \; ( e \circv h ) \circh i
  &\qquad\qquad
  &
  ( ( a \circh ( b \circv g ) ) \circv ( d \circh f ) ) \circh ( c \circv ( ( e \circv h ) \circh i ) )
  \\
  &
  a, \; b \circv g, \; d, \; f
  &\qquad\qquad
  &
  ( a \circv d ) \circh ( b \circv g \circv f ) \circh ( c \circv ( ( e \circv h ) \circh i ) )
  \\
  &
  b, \; g \circv f, \; c, \; ( e \circv h ) \circh i
  &\qquad\qquad
  &
  ( a \circv d ) \circh ( ( b \circh c ) \circv ( ( g \circv f ) \circh ( e \circv h ) \circh i ) )
  \\
  &
  a, \; d, \; b \circh c, \; ( g \circv f ) \circh ( e \circv h ) \circh i
  &\qquad\qquad
  &
  ( a \circh b \circh c ) \circv ( d \circh ( g \circv f ) \circh ( e \circv h ) \circh i )
  \end{alignat*}
  \vspace{-7mm}
  \caption{Proof of Theorem \ref{theorem3981} for component 3981}
  \label{proof3981}
  \end{figure}

\begin{table}[ht]
\begin{tabular}{ccc}
index &\; association type &\; diagram
\\
\\
3989 
&\; 
$( a \circv b ) \circh ( c \circv d \circv e \circv f ) \circh ( g \circv h \circv i )$
&\;
$
\begin{bmatrix}
\begin{bmatrix} a \\ b \end{bmatrix}
&
\begin{bmatrix} c \\ d \\ e \\ f \end{bmatrix}
&
\begin{bmatrix} g \\ h \\ i \end{bmatrix}
\end{bmatrix}
$
\\
\\
3990 
&\; 
$( a \circv b \circv c ) \circh ( d \circv e \circv f \circv g ) \circh ( h \circv i )$
&\;
$
\begin{bmatrix}
\begin{bmatrix} a \\ b \\ c \end{bmatrix}
&
\begin{bmatrix} d \\ e \\ f \\ g \end{bmatrix}
&
\begin{bmatrix} h \\ i \end{bmatrix}
\end{bmatrix}
$
\\
\\
3991 
&\; 
$( a \circh b ) \circv ( c \circh d \circh e \circh f ) \circv ( g \circh h \circh i )$
&\;
$
\begin{bmatrix}
\begin{bmatrix} a & b \end{bmatrix}
\\
\begin{bmatrix} c & d & e & f \end{bmatrix}
\\
\begin{bmatrix} g & h & i \end{bmatrix}
\end{bmatrix}
$
\\
\\
3992 
&\; 
$( a \circh b \circh c ) \circv ( d \circh e \circh f \circh g ) \circv ( h \circh i )$
&\;
$
\begin{bmatrix}
\begin{bmatrix} a & b & c \end{bmatrix}
\\
\begin{bmatrix} d & e & f & g \end{bmatrix}
\\
\begin{bmatrix} h & i \end{bmatrix}
\end{bmatrix}
$
\end{tabular}
\bigskip
\caption{Representative vertices for components of type 2}
\label{table9.2}
\end{table}

\subsection*{Type 2: components 3989 to 3992}

We represent these components by the vertices in Table \ref{table9.2} which have index numbers 
9137, 10898, 30805, 31485 respectively.
Each diagram has a reflection symmetry, so there are four components in this orbit for the action of $D_8$.
Component 3989 has 10 nontrivial basis cycles but none of them contains the representative vertex 9137.
However, there is a directed edge $e$ from vertex 9137 to vertex 14245,
and vertex 14245 belongs to every nontrivial cycle.
We therefore construct nontrivial cycles starting and ending at vertex 9137 by following $e$, 
then the original cycle, and then $e^{-1}$ back to vertex 9137.
(This procedure is admittedly rather artificial and unnatural, and is done for computational convenience
rather than any conceptual reason.)

\begin{theorem} \label{theorem3989}
For component 3989, the nontrivial cycles in the cycle basis produce the following commutativity property,
which transposes the indeterminates $\{ d, e \}$:
\begin{align*}
( a \circv b ) \circh ( c \circv d \circv e \circv f ) \circh ( g \circv h \circv i )
&\equiv
( a \circv b ) \circh ( c \circv e \circv d \circv f ) \circh ( g \circv h \circv i )
\\
\begin{bmatrix}
\begin{bmatrix} a \\ b \end{bmatrix}
&
\begin{bmatrix} c \\ d \\ e \\ f \end{bmatrix}
&
\begin{bmatrix} g \\ h \\ i \end{bmatrix}
\end{bmatrix}
&\equiv
\begin{bmatrix}
\begin{bmatrix} a \\ b \end{bmatrix}
&
\begin{bmatrix} c \\ e \\ d \\ f \end{bmatrix}
&
\begin{bmatrix} g \\ h \\ i \end{bmatrix}
\end{bmatrix}
\end{align*}
Cyclic rotation produces bijections between the four components of type 2
and hence there are corresponding identities in all components.
\end{theorem}

\begin{proof}
We follow the same method as for Theorem \ref{theorem3981}, starting in Figure \ref{proof3989} 
with the monomial on the left side and ending with the monomial on the right side.
\end{proof}

  \begin{figure}[ht]
  \begin{alignat*}{2}
  &
  \text{factors $w, x, y, z$}
  &\qquad\qquad
  &
  \text{result of interchange}
  \\
  &
  &\qquad\qquad
  &
  ( a \circv b ) \circh ( c \circv d \circv e \circv f ) \circh ( g \circv h \circv i )
  \\
  &
  a, \; b, \; c \circv d, \; e \circv f
  &\qquad\qquad
  &
  ( ( a \circh ( c \circv d ) ) \circv ( b \circh ( e \circv f ) ) ) \circh ( g \circv h \circv i )
  \\
  &
  a \circh ( c \circv d ), \; b \circh ( e \circv f ), \; g, \; h \circv i
  &\qquad\qquad
  &
  ( a \circh ( c \circv d ) \circh g ) \circv ( b \circh ( e \circv f ) \circh ( h \circv i ) )
  \\
  &
  e, \; f, \; h, \; i
  &\qquad\qquad
  &
  ( a \circh ( c \circv d ) \circh g ) \circv ( b \circh ( ( e \circh h ) \circv ( f \circh i ) ) )
  \\
  &
  a \circh ( c \circv d ), \; g, \; b, \; ( e \circh h ) \circv ( f \circh i )
  &\qquad\qquad
  &
  ( ( a \circh ( c \circv d ) ) \circv b ) \circh ( g \circv ( e \circh h ) \circv ( f \circh i ) )
  \\
  &
  a \circh ( c \circv d ), \; b, \; g \circv ( e \circh h ), \; f \circh i
  &\qquad\qquad
  &
  ( a \circh ( c \circv d ) \circh ( g \circv ( e \circh h ) ) ) \circv ( b \circh f \circh i )
  \\
  &
  a \circh ( c \circv d ), \; g \circv ( e \circh h ), \; b \circh f, \; i
  &\qquad\qquad
  &
  ( ( a \circh ( c \circv d ) ) \circv ( b \circh f ) ) \circh ( g \circv ( e \circh h ) \circv i )
  \\
  &
  a, \; c \circv d, \; b, \; f
  &\qquad\qquad
  &
  ( a \circv b ) \circh ( c \circv d \circv f ) \circh ( g \circv ( e \circh h ) \circv i )
  \\
  &
  c, \; d \circv f, \; g, \; ( e \circh h ) \circv i
  &\qquad\qquad
  &
  ( a \circv b ) \circh ( ( c \circh g ) \circv ( ( d \circv f ) \circh ( ( e \circh h ) \circv i ) ) )
  \\
  &
  a, \; b, \; c \circh g, \; ( d \circv f ) \circh ( ( e \circh h ) \circv i )
  &\qquad\qquad
  &
  ( a \circh c \circh g ) \circv ( b \circh ( d \circv f ) \circh ( ( e \circh h ) \circv i ) )
  \\
  &
  a, \; c \circh g, \; b \circh ( d \circv f ), \; ( e \circh h ) \circv i
  &\qquad\qquad
  &
  ( a \circv ( b \circh ( d \circv f ) ) ) \circh ( ( c \circh g ) \circv ( e \circh h ) \circv i )
  \\
  &
  a, \; b \circh ( d \circv f ), \; ( c \circh g ) \circv ( e \circh h ), \; i
  &\qquad\qquad
  &
  ( a \circh ( ( c \circh g ) \circv ( e \circh h ) ) ) \circv ( b \circh ( d \circv f ) \circh i )
  \\
  &
  c, \; g, \; e, \; h
  &\qquad\qquad
  &
  ( a \circh ( c \circv e ) \circh ( g \circv h ) ) \circv ( b \circh ( d \circv f ) \circh i )
  \\
  &
  a \circh ( c \circv e ), \; g \circv h, \; b \circh ( d \circv f ), \; i
  &\qquad\qquad
  &
  ( ( a \circh ( c \circv e ) ) \circv ( b \circh ( d \circv f ) ) ) \circh ( g \circv h \circv i )
  \\
  &
  a, \; c \circv e, \; b, \; d \circv f
  &\qquad\qquad
  &
  ( a \circv b ) \circh ( c \circv e \circv d \circv f ) \circh ( g \circv h \circv i )
  \end{alignat*}
  \vspace{-7mm}
  \caption{Proof of Theorem \ref{theorem3989} for component 3989}
  \label{proof3989}
  \end{figure}

\subsection*{Type 3: components 3994 to 4001}

We represent these components by the vertices in Table \ref{table9.3}; their index numbers in 
the total order are 9559, 29656, 36110, 10265, 14904, 16316, 30165, 37128 respectively.
None of the diagrams has any symmetry, so there are eight components in this orbit.
Component 3994 has 3 nontrivial basis cycles but none of them contains the representative vertex 9559.
However, there is a directed path $e_1 e_2$ from 9559 to 9569 to 9771, 
and vertex 9771 belongs to all three cycles.
We therefore construct (again somewhat artificially) nontrivial cycles starting and ending at 9559 
by following $e_1 e_2$, then the cycle, and then returning to 9559 by $e_2^{-1} e_1^{-1}$.

\begin{table}[ht]
\begin{tabular}{ccc}
index & association type & diagram
\\
\\
3994 &  
$( a \circv b ) \circh ( c \circv ( ( d \circv e ) \circh ( f \circv g ) \circh( h \circv i ) ) $ &
$
\begin{bmatrix}
\begin{bmatrix} a \\ b \end{bmatrix}

\begin{bmatrix}
&  c & \\
\begin{bmatrix} d \\ e \end{bmatrix} &  
\begin{bmatrix} f \\ g \end{bmatrix} &  
\begin{bmatrix} h \\ i
\end{bmatrix}
\end{bmatrix}
\end{bmatrix}
$
\\
\\
3995 & 
$( a \circh b ) \circv ( c \circh (  ( d \circh e ) \circv ( f \circh g ) \circv ( h \circh i ) ) ) $ &
$
\begin{bmatrix}
\begin{bmatrix} a &  b \end{bmatrix} \\
\begin{bmatrix}
& \begin{bmatrix} d &  e \end{bmatrix} \\
c  & \begin{bmatrix} f &  g \end{bmatrix} \\
& \begin{bmatrix} h &  i \end{bmatrix}
\end{bmatrix}
\end{bmatrix}
$
\\
\\
3996 & 
$( a \circh ( ( b \circh c ) \circv ( d \circh e ) \circv ( f \circh g ) ) ) \circv ( h \circh i ) $ &
$
\begin{bmatrix}
\begin{bmatrix}
&\begin{bmatrix} b & c \end{bmatrix}\\
a  
&\begin{bmatrix} d &  e \end{bmatrix}\\
&\begin{bmatrix} f &  g \end{bmatrix}
\end{bmatrix}
\\
\begin{bmatrix} h &  i \end{bmatrix}
\end{bmatrix}
$
\\
\\
3997 & 
$( a \circv b ) \circh ( ( ( c \circv d ) \circh ( e \circv f ) \circh ( g \circv h ) ) \circv i ) $ &
$
\begin{bmatrix}
\begin{bmatrix} a \\ b \end{bmatrix}
& 
\begin{bmatrix}
\begin{bmatrix} c \\ d \end{bmatrix} & 
\begin{bmatrix} e \\ f \end{bmatrix} & 
\begin{bmatrix} g \\ h 
\end{bmatrix} \\
& i &
\end{bmatrix}
\end{bmatrix}
$
\\
\\
3998 & 
$( a \circv (  ( b \circv c ) \circh ( d \circv e ) \circh ( f \circv g )  )  ) \circh ( h \circv i ) $ &
$
\begin{bmatrix}
\begin{bmatrix}
&  a & \\
\begin{bmatrix} b \\ c \end{bmatrix} & 
\begin{bmatrix} d \\ e \end{bmatrix} &  
\begin{bmatrix} f \\ g \end{bmatrix}
\end{bmatrix}
& 
\begin{bmatrix} h &  i \end{bmatrix}
\end{bmatrix}
$
\\
\\
3999 & 
$( ( ( a \circv b ) \circh ( c \circv d ) \circh ( e \circv f ) ) \circv g ) \circh ( h \circv i )$ &
$
\begin{bmatrix}
\begin{bmatrix}
\begin{bmatrix} a \\ b \end{bmatrix} &  
\begin{bmatrix} c \\ d \end{bmatrix} &  
\begin{bmatrix} e \\ f 
\end{bmatrix}
\\
& g &
\end{bmatrix}
&
\begin{bmatrix} h \\ i \end{bmatrix}
\end{bmatrix}
$
\\
\\
4000 & 
$ ( a \circh b ) \circv (  (  ( c \circh d ) \circv ( e \circh f ) \circv ( g \circh h )  ) \circh i ) $ &
$
\begin{bmatrix}
\begin{bmatrix} a &  b \end{bmatrix} \\
\begin{bmatrix}
\begin{bmatrix} c &  d \end{bmatrix} & \\ 
\begin{bmatrix} e &  f \end{bmatrix} &  i \\ 
\begin{bmatrix} g &  h \end{bmatrix} &
\end{bmatrix}
\end{bmatrix}
$
\\
\\
4001 & 
$( ( ( a \circh b ) \circv ( ( c \circh d ) \circv ( e \circh f ) ) ) \circh g ) \circv ( h \circh i )$ &
$
\begin{bmatrix}
\begin{bmatrix}
\begin{bmatrix} a &  b \end{bmatrix} & \\
\begin{bmatrix} c &  d \end{bmatrix} &  g \\
\begin{bmatrix} e &  f \end{bmatrix} & \\
\end{bmatrix}
\\
\begin{bmatrix} h &  i \end{bmatrix}
\end{bmatrix}
$
\end{tabular}
\bigskip
\caption{Representative vertices for components of type 3}
\label{table9.3}
\end{table}

\begin{theorem} \label{theorem3994}
For component 3994, the nontrivial cycles in the cycle basis produce the following commutativity property,
which transposes the indeterminates $\{ d, f \}$:
\begin{align*}
( a \circv b ) \circh ( c \circv ( ( d \circv e ) \circh ( f \circv g ) \circh ( h \circv i ) ) )
&\equiv
( a \circv b ) \circh ( c \circv ( ( f \circv e ) \circh ( d \circv g ) \circh ( h \circv i ) ) )
\\
\begin{bmatrix}
\begin{bmatrix} a \\ b \end{bmatrix}
&
\begin{bmatrix}
& c &
\\
\begin{bmatrix} d \\ e \end{bmatrix}
&
\begin{bmatrix} f \\ g \end{bmatrix}
&
\begin{bmatrix} h \\ i \end{bmatrix}
\end{bmatrix}
\end{bmatrix}
&\equiv
\begin{bmatrix}
\begin{bmatrix} a \\ b \end{bmatrix}
&
\begin{bmatrix}
& c &
\\
\begin{bmatrix} f \\ e \end{bmatrix}
&
\begin{bmatrix} d \\ g \end{bmatrix}
&
\begin{bmatrix} h \\ i \end{bmatrix}
\end{bmatrix}
\end{bmatrix}
\end{align*}
The action of the dihedral group $D_8$ induces bijections between the components in type 3,
and so all eight components have corresponding identities.
\end{theorem}

\begin{proof}
Similar to the proofs of Theorems \ref{theorem3981} and \ref{theorem3989}; see Figure \ref{proof3994}.
\end{proof}

  \begin{figure}[ht]
  \begin{alignat*}{2}
  &
  \text{factors $w, x, y, z$}
  &\qquad\qquad
  &
  \text{result of interchange}
  \\
  &
  &\qquad\qquad
  &
  ( a \circv b ) \circh ( c \circv ( ( d \circv e ) \circh ( f \circv g ) \circh ( h \circv i ) ) )
  \\
  &
  f, \; g, \; h, \; i
  &\qquad\qquad
  &
  ( a \circv b ) \circh ( c \circv ( ( d \circv e ) \circh ( ( f \circh h ) \circv ( g \circh i ) ) ) )
  \\
  &
  d, \; e, \; f \circh h, \; g \circh i
  &\qquad\qquad
  &
  ( a \circv b ) \circh ( c \circv ( d \circh f \circh h ) \circv ( e \circh g \circh i ) )
  \\
  &
  a, \; b, \; c, \; ( d \circh f \circh h ) \circv ( e \circh g \circh i )
  &\qquad\qquad
  &
  ( a \circh c ) \circv ( b \circh ( ( d \circh f \circh h ) \circv ( e \circh g \circh i ) ) )
  \\
  &
  d \circh f, \; h, \; e \circh g, \; i
  &\qquad\qquad
  &
  ( a \circh c ) \circv ( b \circh ( ( d \circh f ) \circv ( e \circh g ) ) \circh ( h \circv i ) )
  \\
  &
  a, \; c, \; b \circh ( ( d \circh f ) \circv ( e \circh g ), \; h \circv i
  &\qquad\qquad
  &
  ( a \circv ( b \circh ( ( d \circh f ) \circv ( e \circh g ) ) ) ) \circh ( c \circv h \circv i )
  \\
  &
  a, \; b \circh ( ( d \circh f ) \circv ( e \circh g ) ), \; c \circv h, \; i
  &\qquad\qquad
  &
  ( a \circh ( c \circv h ) ) \circv ( b \circh ( ( d \circh f ) \circv ( e \circh g ) ) \circh i )
  \\
  &
  d, \; f, \; e, \; g
  &\qquad\qquad
  &
  ( a \circh ( c \circv h ) ) \circv ( b \circh ( d \circv e ) \circh ( f \circv g ) \circh i )
  \\
  &
  a, \; c \circv h, \; b \circh ( d \circv e ), \; ( f \circv g ) \circh i
  &\qquad\qquad
  &
  ( a \circv ( b \circh ( d \circv e ) ) ) \circh ( c \circv h \circv ( ( f \circv g ) \circh i ) )
  \\
  &
  a, \; b \circh ( d \circv e ), \; c, \; h \circv ( ( f \circv g ) \circh i )
  &\qquad\qquad
  &
  ( a \circh c ) \circv ( b \circh ( d \circv e ) \circh ( h \circv ( ( f \circv g ) \circh i ) ) )
  \\
  &
  a, \; c, \; b, \; ( d \circv e ) \circh ( h \circv ( ( f \circv g ) \circh i ) )
  &\qquad\qquad
  &
  ( a \circv b ) \circh ( c \circv ( ( d \circv e ) \circh ( h \circv ( ( f \circv g ) \circh i ) ) ) )
  \\
  &
  d, \; e, \; h, \; ( f \circv g ) \circh i
  &\qquad\qquad
  &
  ( a \circv b ) \circh ( c \circv ( d \circh h ) \circv ( e \circh ( f \circv g ) \circh i ) )
  \\
  &
  a, \; b, \; c \circv ( d \circh h ), \; ( e \circh ( f \circv g ) \circh i )
  &\qquad\qquad
  &
  ( a \circh ( c \circv ( d \circh h ) ) ) \circv ( b \circh e \circh ( f \circv g ) \circh i )
  \\
  &
  a, \; c \circv ( d \circh h ), \; b \circh e \circh ( f \circv g ), \; i
  &\qquad\qquad
  &
  ( a \circv ( b \circh e \circh ( f \circv g ) ) ) \circh ( c \circv ( d \circh h ) \circv i )
  \\
  &
  a, \; b \circh e \circh ( f \circv g ), \; c, \; ( d \circh h ) \circv i
  &\qquad\qquad
  &
  ( a \circh c ) \circv ( b \circh e \circh ( f \circv g ) \circh ( ( d \circh h ) \circv i ) )
  \\
  &
  f, \; g, \; d \circh h, \; i
  &\qquad\qquad
  &
  ( a \circh c ) \circv ( b \circh e \circh ( ( f \circh d \circh h ) \circv ( g \circh i ) ) )
  \\
  &
  a, \; c, \; b \circh e, \; ( f \circh d \circh h ) \circv ( g \circh i )
  &\qquad\qquad
  &
  ( a \circv ( b \circh e ) ) \circh ( c \circv ( f \circh d \circh h ) \circv ( g \circh i ) )
  \\
  &
  a, \; b \circh e, \; c \circv ( f \circh d \circh h ), \; g \circh i
  &\qquad\qquad
  &
  ( a \circh ( c \circv ( f \circh d \circh h ) ) ) \circv ( b \circh e \circh g \circh i )
  \\
  &
  a, \; c \circv ( f \circh d \circh h ), \; b, \; e \circh g \circh i
  &\qquad\qquad
  &
  ( a \circv b ) \circh ( c \circv ( f \circh d \circh h ) \circv ( e \circh g \circh i ) )
  \\
  &
  f, \; d \circh h, \; e, \; g \circh i
  &\qquad\qquad
  &
  ( a \circv b ) \circh ( c \circv ( ( f \circv e ) \circh ( ( d \circh h ) \circv ( g \circh i ) ) ) )
  \\
  &
  d, \; h, \; g, \; i
  &\qquad\qquad
  &
  ( a \circv b ) \circh ( c \circv ( ( f \circv e ) \circh ( d \circv g ) \circh ( h \circv i ) ) )
  \end{alignat*}
  \vspace{-7mm}
  \caption{Proof of Theorem \ref{theorem3994} for component 3994}
  \label{proof3994}
  \end{figure}


\section{Conclusion} \label{sectionconclusion}

\subsection*{Group structure on commutativity properties}

For each connected component in each degree, the commutativity properties form a group
whose multiplication is induced by the concatenation of cycles.
For the components of $G(9)$ discussed in the previous section, the fundamental group is the symmetric group 
on two letters.  
For Kock's relation \eqref{kockidentity} in degree 16, the fundamental group is the symmetric group $S_4$, 
which permutes the four inner boxes in the geometric realization; this follows from the dihedral symmetry of 
the geometric realization and the fact that $S_4$ is generated by the transpositions (1,2), (2,3), (3,4).
The geometry of the commutativity properties suggests that braiding occurs: 
there is a difference between transposing $a, b$ twice and doing nothing, since in the former case $a$ makes 
a complete circuit around $b$.  
For the components of $G(9)$, this implies that the fundamental group $\pi(G,t)$ is not simply the symmetric 
group on two letters (cyclic of order 2) but rather the braid group on two letters (infinite cyclic).

\subsection*{Higher commutativity properties}

Our colleague Michael Kinyon used Prover9, the automated theorem prover \cite{McCune}, 
to find a simple derivation of Kock's relation from the relation of our Theorem \ref{theorem3989}:
\[
( a \circv b ) \circh ( c \circv d \circv e \circv f ) \circh ( g \circv h \circv i )
\equiv
( a \circv b ) \circh ( c \circv e \circv d \circv f ) \circh ( g \circv h \circv i ).
\]
If we make the substitutions $b \leftarrow b \circv j \circv k$ and $i \leftarrow i \circv \ell$, and 
right-multiply both sides by $m \circv n \circv p \circv q$ using the operation $\circh$,
then we obtain a relation which coincides with \eqref{kockidentity} after permuting the 
indeterminates and transposing $\circh$ and $\circv$:
\begin{align*}
&
( a \circv b \circv j \circv k ) \circh 
( c \circv d \circv e \circv f ) \circh 
( g \circv h \circv i \circv \ell ) \circh
( m \circv n \circv p \circv q )
\equiv
\\
&
( a \circv b \circv j \circv k ) \circh 
( c \circv e \circv d \circv f ) \circh 
( g \circv h \circv i \circv \ell ) \circh
( m \circv n \circv p \circv q ).
\end{align*}

An important problem is to determine whether there exist commutativity properties in degree $n \ge 10$ 
which are not consequences of the known identities for $n = 9$.
More generally, we would like to find a complete set of independent generators for all commutativity 
properties of double semigroups in all degrees.

There are analogous problems for $d$-tuple semigroups with $d \ge 3$; 
by definition, these have $d$ associative (binary) operations with each pair of operations satisfying 
the interchange identity \eqref{interchange}.
The combinatorial problem of determining the number of association types in degree $n$ for $d$-tuple 
semigroups leads to a $d$-dimensional generalization of the Schr\"oder numbers which is related 
to the enumeration of guillotine partitions in $d$ dimensions.
See Ackerman et al.~\cite{ABPR2006} and Asinowski et al.~\cite{ABMP2014}, as well as sequence A103209 
in the OEIS.

\subsection*{Consequences for algebraic operads}

We conclude with an application of our results to interchange algebras.

\begin{definition} (Loday and Vallette \cite[\S 13.10.4]{LodayVallette2012})
An \textbf{interchange algebra} is a vector space $A$
over a field $\mathbb{F}$ with two bilinear products $\circh, \circv \colon A \times A \to A$
which are associative and satisfy the interchange relation \eqref{interchange}.
The corresponding symmetric operad is the \textbf{interchange operad}; we denote it by $\mathsf{INT}$.
\end{definition}

The operad $\mathsf{INT}$ is binary (the operations are bilinear), cubic (the identities involve at most 
three operations),
and symmetric (the interchange relation involves a nontrivial permutation of the variables).
Since $\mathsf{INT}$ is not quadratic, there does not exist a Koszul dual.
In arity $n$, the $S_n$-module $\mathsf{INT}(n)$ consists of all $n$-ary multilinear operations generated by 
$\circh, \circv$ 
subject to associativity and the interchange relation; this $S_n$-module is isomorphic to the 
multilinear subspace of 
degree $n$ in the free interchange algebra on $n$ generators.

Even though $\mathsf{INT}$ is a symmetric operad, it is a \textsf{Set} operad in the sense that 
all of its defining relations have the form $x \equiv y$ for two monomials $x, y$: the relations
can be expressed without using the vector space structure.
Therefore the space of $n$-ary operations in $\mathsf{INT}$ is the vector space spanned by 
the set of $n$-ary operations for the corresponding operad in \textsf{Set}, namely the free 
double semigroup on one generator.

The two-associative operad $\mathsf{AA}$ is the non-symmetric operad with two associative binary operations 
and no further relations; it satisfies $\dim \mathsf{AA}(n) = T(n)$ since by Lemma \ref{lemmaschroder} 
we know that $| VG(n) | = T(n)$.
In the corresponding symmetric operad $\mathsf{SAA}$, the $S_n$-module in arity $n$ is the direct sum of 
$T(n)$ copies of the left regular representation $\mathbb{F} S_n$, and so $\dim \mathsf{SAA}(n) = T(n) n!$.
The interchange operad $\mathsf{INT}$ is the quotient of $\mathsf{SAA}$ by the operad ideal generated
by the interchange relation.

If two association types $t_1, t_2 \in VG(n)$ belong to the same connected component of $G(n)$, then $t_1$ 
can be transformed into $t_2$ by a sequence of consequences of the interchange relation, 
and hence the two copies of $\mathbb{F} S_n$ corresponding to $t_1, t_2$ will become identified
in the quotient operad $\mathsf{INT}$.
Thus the $S_n$-module $\mathsf{INT}(n)$ will be no larger than the direct sum of $U(n)$ copies of 
$\mathbb{F} S_n$, where $U(n)$ is the number of connected components of $G(n)$ including isolated vertices.

If the monodromy group is trivial for some connected component of $G(n)$, then that component corresponds
to a direct summand isomorphic to $\mathbb{F} S_n$ in the quotient module $\mathsf{INT}(n)$.
The reason is that since there are no syzygies (relations among the relations) obtained from the consequences 
of the interchange relation, this component corresponds to a free summand in the $S_n$-module $\mathsf{INT}(n)$.

However, for any association type $t$ which is a vertex in a component of $G(n)$ admitting a nontrivial 
cycle corresponding to a commutativity relation $m \equiv m'$ in double semigroups, 
there will be a relation in the operad ideal generated by the interchange relation which implies that
the corresponding copy of $\mathbb{F} S_n$ will reduce to its quotient modulo the left $S_n$-module
generated by the commutativity property expressed in the linear form $m - m' \equiv 0$.

Thus our results imply that $\mathsf{INT}(n)$ is not a free $S_n$-module for $n \ge 9$.


\section*{Acknowledgements}

We thank the referee for a very detalied report which led to significant improvements in many parts 
of the paper.
In particular, the referee pointed out that the natural setting for our results is the covering map 
$p\colon F(n) \to G(n)$, and emphasized the connection between the nontrivial monodromy of this map 
and the structure of the groupoid generated by $G(n)$.
The referee also brought to our attention the fact that for each connected component in each degree, 
the commutativity properties have a natural group structure induced by the concatenation of cycles.

Murray Bremner was supported by a Discovery Grant from NSERC, the Natural Sciences and Engineering 
Research Council of Canada.
Sara Madariaga was supported by a Postdoctoral Fellowship from PIMS, the Pacific Institute for the 
Mathematical Sciences.



\begin{thebibliography}{99}

\bibitem{Acerbi2003}
\textsc{F. Acerbi}:
On the shoulders of Hipparchus: a reappraisal of ancient Greek combinatorics.
\emph{Arch. Hist. Exact Sci.}
57 (2003), no.~6, 465--502. 

\bibitem{ABPR2006}
\textsc{E. Ackerman, G. Barequet, R. Pinter, D. Romik}:
The number of guillotine partitions in $d$ dimensions.
\emph{Inform. Process. Lett.} 
98 (2006), no.~4, 162--167. 

\bibitem{ABMP2014}
\textsc{A. Asinowski, G. Barequet, T. Mansour, R. Pinter}:
Cut equivalence of $d$-dimensional guillotine partitions.
\emph{Discrete Math.} 
331 (2014) 165--174. 

\bibitem{Berge2001}
\textsc{C. Berge}:
\emph{The Theory of Graphs}.
Translated from the 1958 French edition by Alison Doig.
Second printing of the 1962 first English edition.
Dover Publications, Mineola, NY, 2001.

\bibitem{Brown1982}
\textsc{R. Brown}:
Higher-dimensional group theory.
\emph{Low-dimensional Topology (Bangor, 1979)},
pp.~215--238.
London Math. Soc. Lecture Note Ser., 48.
Cambridge University Press, 1982.

\bibitem{Brown1987}
\textsc{R. Brown}:
From groups to groupoids: a brief survey.
\emph{Bull. London Math. Soc.}
19 (1987), no.~2, 113--134.

\bibitem{Coker2004}
\textsc{C. Coker}:
A family of eigensequences.
\emph{Discrete Math.} 
282 (2004), no.~1-3, 249--250. 

\bibitem{Deutsch2001}
\textsc{E. Deutsch}:
A bijective proof of the equation linking the Schr\"oder numbers, large and small.
\emph{Discrete Math.} 
241 (2001), no.~1-3, 235--240.

\bibitem{DeWolf2013}
\textsc{D. DeWolf}:
\emph{On Double Inverse Semigroups}.
Master's Thesis.
Dalhousie University, Nova Scotia, Canada, 2013, 93 pages.

\bibitem{EckmannHilton1961}
\textsc{B. Eckmann, P. J. Hilton}:
Group-like structures in general categories. I. Multiplications and comultiplications.
\emph{Math. Ann.}
145 (1961/1962) 227--255.

\bibitem{FoataZeilberger1997}
\textsc{D. Foata, D. Zeilberger}:
A classic proof of a recurrence for a very classical sequence.
\emph{J. Combin. Theory Ser. A} 
80 (1997), no.~2, 380--384. 

\bibitem{HKL1998}
\textsc{L. Habsieger, M. Kazarian, S. Lando}:
On the second number of Plutarch.
\emph{Amer. Math. Monthly} 
105 (1998), no.~5, 446. 

\bibitem{Kock2007}
\textsc{J. Kock}:
Note on commutativity in double semigroups and two-fold monoidal categories.
\emph{J. Homotopy Relat. Struct.}
2 (2007), no.~2, 217--228.

\bibitem{LodayVallette2012}
\textsc{J.-L. Loday, B. Vallette}:
\emph{Algebraic Operads}.
Grundlehren der mathematischen Wissenschaften, 346.
Springer, Heidelberg, 2012.

\bibitem{MacLane1963}
\textsc{S. MacLane}:
Natural associativity and commutativity.
\emph{Rice Univ. Studies} 
49 (1963), no.~4, 28--46. 

\bibitem{MacLane1998}
\textsc{S. MacLane}:
\emph{Categories for the Working Mathematician}.
Second edition.
Graduate Texts in Mathematics, 5.
Springer, New York, 1998.

\bibitem{McCune}
\textsc{W. McCune}:
Prover9, an automated theorem prover for first-order and equational logic.
\url{www.cs.unm.edu/~mccune/mace4/}

\bibitem{PP2006}
\textsc{R. Padmanabhan, R. Penner}:
An implication basis for linear forms.
\emph{Algebra Universalis}
55 (2006), no. 2-3, 355--368.

\bibitem{Schroder1870}
\textsc{E. Schr\"oder}:
Vier kombinatorische Probleme.
\emph{Zeitschrift f\"ur Mathematik und Physik}
XV (1870) 361--376.
\url{gdz.sub.uni-goettingen.de/dms/load/toc/?PPN=PPN599415665}

\bibitem{ShapiroSulanke2000}
\textsc{L. W. Shapiro, R. A. Sulanke}:
Bijections for the Schr\"oder numbers.
\emph{Math. Mag.}
73 (2000), no.~5, 369--376.

\bibitem{Simpson2012}
\textsc{C. Simpson}:
\emph{Homotopy Theory of Higher Categories}.
New Mathematical Monographs, 19. 
Cambridge University Press, 2012.

\bibitem{Stanley1997}
\textsc{R. P. Stanley}:
Hipparchus, Plutarch, Schr\"oder, and Hough.
\emph{Amer. Math. Monthly}
104 (1997), no.~4, 344--350.

\bibitem{Stanley1999}
\textsc{R. P. Stanley}:
\emph{Enumerative Combinatorics. Volume 2}.
Cambridge Studies in Advanced Mathematics, 62. 
Cambridge University Press, 1999. 

\bibitem{Zinbiel2012}
\textsc{G. W. Zinbiel (J.-L. Loday)}:
Encyclopedia of types of algebras 2010.
\emph{Operads and Universal Algebra}, pp.~217--297.
Nankai Ser. Pure Appl. Math. Theoret. Phys., 9. 
World Sci. Publ., Hackensack, NJ, 2012. 

\end{thebibliography}
\end{document}